\documentclass[oneside,11pt]{article}
\usepackage[top=1.25in, bottom=1.25in, left=0.8in, right=0.8in]{geometry}
\usepackage{amssymb}
\usepackage{amsthm}
\usepackage{graphicx}
\usepackage{setspace}
\usepackage{titling}
\usepackage[utf8]{inputenc}
\usepackage[english]{babel}
\usepackage{mathrsfs}
\usepackage{amsfonts}
\usepackage{amsmath}
\usepackage{float}
\usepackage{color}
\usepackage{dsfont}
\usepackage[toc,page]{appendix}
\usepackage{eqnarray}
\usepackage{bbm}
\usepackage{mathtools}
\usepackage{verbatim}

\usepackage{accents}

\newcommand\munderbar[1]{%
  \underaccent{\bar}{#1}}

\newtheorem{theorem}{Theorem}
\newtheorem{proposition}{Proposition}
\newtheorem{lemma}{Lemma}
\newtheorem{definition}{Definition}
\newtheorem{remark}{Remark}
\newtheorem{corollary}{Corollary}
\newtheorem{hypothesis}{Hypothesis}

\def\balpha{\boldsymbol{\alpha}}

\def\balpha{\boldsymbol{\alpha}}

\def\bphi{\boldsymbol{\phi}}

\def\cG{\mathcal G}
\def\cH{\mathcal H}

\def\cP{\mathcal P}
\def\cT{\mathcal T}

\begin{document}

\renewenvironment{proof}{\noindent{\bf Proof.}\hspace*{1em}}{\qed\bigskip\\}
\newenvironment{proof-sketch}{\noindent{\bf Sketch of Proof.}\hspace*{1em}}{\qed\bigskip\\}

\title{Finite State Mean Field Games with Major and Minor Players}
\author{Rene Carmona and Peiqi Wang\\
Department of Operations Research and Financial Engineering\\
Princeton University}
\maketitle

\begin{abstract}
The goal of the paper is to develop the theory of finite state mean field games with major and minor players when the state space of the game is finite. We introduce the finite player games and derive a mean field game formulation in the limit when the number of minor players tends to infinity. In this limit, we prove that the value functions of the optimization problems are viscosity solutions of PIDEs of the HJB type, and we construct the best responses for both types of players. From there, we prove existence of Nash equilibria under reasonable assumptions.
Finally we prove that a form of propagation of chaos holds in the present context and use this result to prove existence of approximate Nash equilibria for the finite player games from the solutions of the mean field games. this vindicate our formulation of the mean field game problem.
\end{abstract}

\section{Introduction}
Mean field games with major and minor players were introduced to accommodate the presence of subgroups of players whose influence on the behavior of the remaining population does not vanish in the asymptotic regime of large games. In this paper we develop the theory of
these dynamic games  when the optimization problems faced by the players are over the dynamics of continuous time 
controlled processes with values in finite state spaces. The theory of finite state mean field games for a single homogeneous population of players was introduced in \cite{Gueant_tree,Gueant_congestion} and  \cite{GomesMohrSouza_continuous}. The interested reader may also consult Chapter 7 of the book \cite{CarmonaDelarue_book_I} for a complete presentation of the theory. 
The present paper is concerned with the extension to models with major and minor players. We search for closed loop Nash equilibria and for this reason, we use the approach which was advocated in \cite{CarmonaWang_LQ}, and called \emph{an alternative approach} in Chapter 13 of \cite{CarmonaDelarue_book_II}. 

\vskip 2pt
Our interest in mean field games with major and minor players when the state space is finite was sparked by the four state model \cite{KolokoltsovBensoussan} for the behavior of computer owners facing cyber attacks. Even though the model was not introduced and treated as a game with major and minor players, clearly, it is of this type if the behaviors of the attacker and the targets are strategic.
Practical applications amenable to these models abound and a better theoretical understanding of their structures should lead to sorely needed numerical procedures to compute Nash equilibria.

Early forms of mean field games with major and minor players appeared in \cite{Huang} in an infinite-horizon setting, in \cite{NguyenHuang1} for finite time horizons, and \cite{NourianCaines} offered a first generalization to non linear-quadratic cases. In these models, the state of the major player does not enter the dynamics of the states of the minor players: it only appears in their cost functionals. this was remedied in \cite{NguyenHuang2}  for linear quadratic models. 
The asymmetry between major and minor players was emphasized in \cite{BensoussanChauYam}
where the authors insist on the fact that the statistical distribution of the state  of a generic minor player should be derived endogenously.  Like \cite{NourianCainesMalhame}, \cite{BensoussanChauYam} characterizes the limiting problem by a set of stochastic partial differential equations. 
However, \cite{BensoussanChauYam} seems to be solving a Stackelberg game, and only the population of minor players ends up in a Nash equilibrium. this is in contrast with \cite{CarmonaZhu} which also insists on the endogenous nature of the statical distribution of the state of a generic minor player, but which formulates the search for a mean field equilibrium as the search for a Nash equilibrium in a two player game 
over the time evolutions of states, some of which being of a McKean-Vlasov type. The recent technical report \cite{JaimungalNourian} adds a major player to the particular case (without idiosyncratic random shocks) of extended mean field game model of optimal execution introduced in Chapter 1 and solved in  Chapter 4 of \cite{CarmonaDelarue_book_I}.

In this paper, we cast the search for Nash equilibria as a search for fixed points of the best response function constructed from the optimization problems of both types of players. 
Typically, in a mean field game with major and minor players, the dynamics of the state $X^0_t$ of the major player (as well as its costs) depend upon the statistical distribution $\mu_t$ of the state $X_t$ of a generic minor player. Throughout the paper we consider that the players are gender neutral and we use "its" instead of "his" or "her". Alternatively,  the dynamics of the state $X_t$ of a generic minor player (as well as its costs) depend upon the values of the state $X^0_t$ and the control $\alpha^0_t$ of the major player as well as the statistical distribution $\mu_t$ which captures the mean field interactions between the minor players. In this paper, we prove that he processes $(X_t^0, \mu_t)$ and $(X_t^0, X_t, \mu_t)$ are Markovian and we characterize their laws by their infinitesimal generators.
We start from the finite player version of the model and show convergence when the number of minor players goes to infinity. We rely on standard results on the convergence of Markov semigroups.
Note that the control of the major player implicitly influences $\mu_t$ through the major player's state, so the major player's optimization problem should be treated as an optimal control problem for McKean-Vlasov dynamics. On the other hand, for the representative minor player's problem, we are just dealing with a classical Markov decision problem in continuous time.
this allows us to adapt to the finite state space the approach introduced in \cite{CarmonaWang_LQ} and reviewed in Chapter 13 of \cite{CarmonaDelarue_book_II}, to define and construct Nash equilibria. We emphasize that these are Nash equilibria for the whole system \emph{Major + Minor Players} and not only for the minor players. this is fully justified by our results on the propagation of chaos and their applications to the proof that our mean field game equilibria provide approximate Nash equilibria for finite games, including both major and minor players.

\vskip 4pt
The paper is structured as follows. Games with finitely many minor players and a major player are introduced in Section \ref{se:finite} where we explain the conventions and notations we use to describe continuous time controlled Markov processes in finite state spaces. We also identify the major and minor players by specifying the information structures available to them, the types of actions they can take, and the costs they incur.
The short and non-technical Section \ref{se:mfg} describes the mean field game strategy and emphasizes the steps needed in the search for Nash equilibria for the system. this is in contrast with some earlier works where the formulation of the problem lead to Stackelberg equilibria, only the minor players being in an approximate Nash equilibrium. To keep with the intuition that the mean field game strategy is to implement a form of limit when the number of minor players grows to infinity, Section \ref{se:convergence} considers the convergence of the state Markov processes in this limit, and identifies the optimization problems which the major and minor players need to solve in order to construct their best responses. this leads to the formalization of the search for a mean field equilibrium as the search of fixed points for the best response map so constructed. The optimization problems underpinning the definition of the best response map are studied in Section \ref{se:optimizations}. There, we use dynamic programming to prove that the value functions of these optimization problems are viscosity solutions of HJB type Partial Integro - Differential Equations (PIDEs). Section \ref{se:Nash} proves existence of the best response map and of Nash equilibria under reasonable conditions. Next, Section \ref{se:master} gives a verification theorem based on the existence of a classical solution to the master equation.
The longer Section \ref{se:chaos} proves that the solution of the mean field game problem provides approximate Nash equilibria for the finite player games. this vindicates our formulation as the right formulation of the problem if the goal is to find Nash equilibria for the system including both major and minor players. The strategy of the proof in by now standard in the literature on mean field games. It relies on propagation of chaos results. However, the latter are usually derived for stochastic differential systems with mean field interactions, and because we could not find the results we needed in the existing literature, we provide proofs of the main steps of the derivations of these results in the context of controlled Markov evolutions in finite state spaces. Finally, an appendix provides the proofs of some of the technical results we used in the text.

\section{Game Model with Finitely Many Players}
\label{se:finite}

We consider a stochastic game in continuous time, involving a major player indexed by $0$, and $N$ minor players indexed from $1$ to $N$. The states of all the players $X_t^0, X_t^1, \dots, X_t^N$ are described by a continuous-time finite-state Markov process. Let us denote $\{1,2,\dots,M^0\}$ the set of possible states of the major player, and $\{1,2,\dots,M\}$ the set of possible states of the minor players. We introduce the empirical distribution of the states of the minor players at time $t$:
\[
\mu_t^N = [\frac{1}{N}\sum_{n=1}^N \mathbbm{1}(X_t^n = 1), \frac{1}{N}\sum_{n=1}^N \mathbbm{1}(X_t^n = 2), \dots, \frac{1}{N}\sum_{n=1}^N \mathbbm{1}(X_t^n = M-1)]
\]
We denote by $\mathcal{P}$ the $(M-1)$ - dimensional simplex:
\[
\mathcal{P} := \{x\in\mathbb{R}^{M-1} | x_i \ge 0, \sum x_i \le 1\}.
\]
Obviously, $\mu_t^N\in\cP$. We consider continuous-time Markov dynamics according to which the rates of jump, say $q$, of the state of a generic minor player depends upon the value of its control, the empirical distribution of the states of all the minor players, as well as the major player's control and state. We denote by $A_0$ (resp. $A$) a convex set in which the major player (resp. all the minor players) can choose their controls. So we introduce a function $q$:
\[
[0,T]\times \{1,\dots, M\} ^2 \times A\times \{1, \dots, M^0\}\times A_0 \times \mathcal{P} \ni (t,i,j,\alpha,i^0,\alpha^0,x)\rightarrow  q(t,i,j,\alpha, i^0, \alpha^0, x)
\]
and we make the following assumption on $q$:

\begin{hypothesis}
\label{transrateminor}
For all $(t, \alpha, i^0,\alpha^0, x) \in [0,T]\;\times\;A\;\times\;\{1,\dots, M^0\}\;\times\;A_0\;\times\;\mathcal{P}$, the matrix $[q(t,i,j,\alpha, i^0, \alpha^0,x)]_{1\le i,j\le M}$ is a Q-matrix.
\end{hypothesis}
\noindent
Recall that a matrix $Q=[Q(i,j)]_{i,j}$ is said to be a Q-matrix if $Q(i,j)\ge 0$ for $i\ne j$ and 
$$
\sum_{j \neq i} Q(i,j) = -Q(i,i),\qquad \text{for all } i.
$$

\noindent Then we assume that at time $t$ if the state of minor player $n$ is $i$, this state will jump from $i$ to $j$ at a rate given by:
\[
q(t, i, j, \alpha^n_t, X_t^0, \alpha^0_t, \mu_t^N)
\]
if $X_t^0$ is the major player state, $\alpha^0_t\in A_0$ is the major player control, $\alpha_t^n \in A$ is the $n$-th minor player control and $\mu_t^N \in\mathcal{P}$ is the empirical distribution of the minor player's states. Our goal is to use these rates to completely specify the law of a continuous time  process in the following way: if at time $t$ the $n$-th minor player is in state $i$ and uses control $\alpha^n_t$, if the major player is in state $X_t^0$ and uses the control $\alpha^0_t$, and if the empirical distribution of the states of the population of minor players is $\mu_t^N$, then the probability of player $n$ remaining in the same state during the infinitesimal time interval $[t,t+\Delta t)$ is $[1 + q(t,i,i,\alpha_t^n, X_t^0,\alpha^0_t, \mu_t^N) \Delta t + o(\Delta t)]$, whereas the probability of this state changing to another state $j$ during the same time interval is given by $[q(t,i,j,\alpha_t^n, X_t^0,\alpha^0_t, \mu_t^N) \Delta t + o(\Delta t)]$.

\vskip 4pt
Similarly, to describe the evolution of the state of the major player we introduce a function $q^0$:
\[
[0,T]\times \{1,\dots, M^0\} ^2 \times A_{0}\times \mathcal{P} \ni (t,i^0,j^0,\alpha^0,x)\rightarrow  q^0(t,i^0,j^0,\alpha^0, x)
\]
which satisfies the following assumption:
\begin{hypothesis}\label{transratemajor}
For each $(t,\alpha^0, x) \in [0,T]\times A_{0}\times\mathcal{P}$, $[q^0(t,i^0,j^0,\alpha^0, x)]_{1\le i^0,j^0\le M^0}$ is a Q-matrix.
\end{hypothesis}

\noindent So if at time $t$ the state of the major player is $i^0$, its control is $\alpha^0_t\in A_0$, and the empirical distribution of the states of the minor players is $\mu_t^N$, we assume the state of the major player will jump to state $j^0$ at rate $q^0(t, i^0, j^0, \alpha^0_t, \mu_t^N)$.

\vskip 4pt
We now define the  control strategies which are admissible to the major and minor players. In our model, we assume that the major player can only observe its own state and the empirical distribution of the states of the minor players, whereas each minor player can observe its own state, the state of the major player as well as the empirical distribution of the states of all the minor players. Furthermore, we only allow for Markov strategies given by feedback functions. Therefore the major player's control should be of the form $\alpha^0_t = \phi^0(t, X_t^0, \mu_t^N)$ for some feedback function $\phi^0:[0,T]\times\{1,\cdots,M^0\}\times\cP\mapsto A_0$, and the control of minor player $n$ should be of the form $\alpha_t^n = \phi^n(t, X_t^n, X_t^0. \mu_t^N)$ for some feedback function $\phi^n:[0,T]\times\{1,\cdots,M\}\times\{1,\cdots,M^0\}\times\cP\mapsto A$. We denote the sets of admissible control strategies by $\AA^0$ and $\AA^n$ respectively. Depending upon the application, we may add more restrictive conditions to the definitions of these sets of admissible control strategies.

\vspace{3mm}
We now define the joint dynamics of the states of all the players. We assume that conditioned on the current state of the system, the changes of states are independent for different players. this means that for all $i^0, i^1, \dots, i^N$ and $j^0, j^1, \dots, j^N$, where $i^0, j^0 \in \{1,2,\dots, M^0\}$ and $i^n, j^n \in \{1,2,\dots, M\}$ for $n=1,2,\dots, N$, we have:
\begin{align*}
&\mathbb{P}[X_{t+\Delta t}^0 = j^0, X_{t+\Delta t}^1 = j^1, \dots, X_{t+\Delta t}^N = j^N | X_t^0 = i^0, X_t^1 = i^1, \dots, X_t^N = i^N]\\
&\hskip 35pt
:= [\mathbbm{1}_{i^0 = j^0} + q^0(t, i^0, j^0, \phi^0(t, i^0, \mu_t^N), \mu_t^N)\Delta t + o(\Delta t)]\\
&\hskip 55pt
 \times \prod_{n=1}^N [\mathbbm{1}_{i^n = j^n} + q^n(t, i^n, j^n, \phi^n(t, i^n, i^0, \mu_t^N), i^0, \phi^0(t,i^0, \mu_t^N), \mu_t^N)\Delta t + o(\Delta t)]
\end{align*}
Formally, this statement is equivalent to the definition of the Q-matrix , say $Q^{(N)}$ of the continuous-time Markov chain $(X_t^0, X_t^1, X_t^2, \dots X_t^N)$. The state space of this Markov chain is the Cartesian product of each player's state space. Therefore $Q^{(N)}$ is a square matrix of size $M^0 \cdot M^N$. The non-diagonal entry of $Q^{(N)}$ can be found by simply retaining the first order term in $\Delta t$ when expanding the above product of probabilities. Because we assume that the transitions of states are independent among the individual players, $Q^{(N)}$ is a sparse matrix.

Each individual player aims to minimize its expected cost in the game. We assume that these costs are given by:
\begin{align*}
J^{0,N}(\alpha^0, \alpha^1,\dots, \alpha^N ) :=\;\;&\; \mathbb{E}\left[\int_{0}^{T}f^0(t,X_t^0, \phi^0(t, X_t^0, \mu_t^N),  \mu_t^N) dt + g^0(X_T^0, \mu_T^N)\right]\\
J^{n,N}(\alpha^0, \alpha^1,\dots, \alpha^N ) :=\;\;&\; \mathbb{E}\left[\int_{0}^{T}f^n(t, X_t^n,  \phi^n(t, X_t^n, X_t^0, \mu_t^N),X_t^0, \phi^0(t, X_t^0, \mu_t^N),  \mu_t^N) dt + g^n(X_T^n, X_T^0, \mu_T^N)\right].
\end{align*}
In this paper, we focus on the special case of symmetric games, for which all the minor players share the same transition rate function and cost function, i.e.: $q^n:=q$, $f^n:=f$, $g^n:=g$, $J^{n,N}:=J^N$, and we search for symmetric Nash equilibria. We say that a couple of feedback functions $(\phi^0,\phi)$ form a symmetric Nash equilibrium if the controls $(\balpha^0,\balpha^1,\cdots,\balpha^N)$ given by $\alpha^0_t=\phi^0(t,X^0_t,\mu^N_t)$ and 
$\alpha^n_t=\phi(t,X^n_t,X^0_t,\mu^N_t)$ for $n=1,\cdots, N$, form a Nash equilibrium in the sense that:
\begin{align*}
J^{0,N}(\balpha^0, \balpha^1,\dots, \balpha^N) \le\;\;& J^{0,N}(\balpha', \balpha^1,\dots, \balpha^N)\\
J^{N}(\balpha^0, \balpha^1,\dots, \balpha^n, \dots \balpha^N) \le\;\;& J^{N}(\balpha^0, \balpha^1,\dots, \balpha^{'n}, \dots \balpha^N)
\end{align*}
for any choices of alternative admissible controls $\balpha^{'0}$ and $\balpha^{'n}$ of the forms $\alpha^{'0}_t=\phi^{'0}(t,X^0_t,\mu^N_t)$ and 
$\alpha^{'n}_t=\phi'(t,X^n_t,X^0_t,\mu^N_t)$.
In order to simplify the notation, we will systematically use the following notations when there is no risk of possible confusion. When $\balpha^0\in\mathbb{A}^0$ is given by a feedback function $\phi^0$ and $\balpha\in\mathbb{A}$ is given by a feedback fucntion $\phi$, we denote by $q^0_{\phi^0}$, $q_{\phi^0,\phi}$, $f^0_{\phi^0}$ and $f_{\phi^0,\phi}$ the functions:
\begin{align*}
q^{0}_{\phi^0}(t,i^0,j^0,x) :=&\;\; q^0(t,i^0,j^0, \phi^0(t,i^0,x),x)\\
q_{\phi^0,\phi}(t,i,j,i^0,x) :=&\;\; q(t,i,j,\phi(t,i,i^0,x),i^0,\phi^0(t,i^0,x), x)\\
f^0_{\phi^0}(t,i^0,x) := &\;\;f^0(t,i^0, \phi^0(t,i^0,x),x)\\
f_{\phi^0,\phi}(t,i^0,i,x) := &\;\;f(t,i,\phi(t,i,i^0,x),i^0,\phi^0(t,i^0,x),x)
\end{align*}

\section{Mean Field Game Formulation}
\label{se:mfg}
Solving for Nash equilibria when the number of players is finite is challenging. There are many reasons why the problem becomes quickly intractable. Among them is the fact that as the number of minor players increases, the dimension of the Q - matrix of the system increases exponentially. The paradigm of Mean Field Games consists in the analysis of the limiting case where the number $N$ of minor players  tends to infinity.  In this asymptotic regime, one expects that simplifications due to averaging effects will make it easier to find asymptotic solutions which could provide approximative equilibria for finite player games when the number $N$ of minor players is large enough. The rationale for such a belief is based on the intuition provided by classical results on the propagation of chaos for large particle systems with mean field interactions.
We developed these results later in the paper.

\vskip 4pt
The advantage of considering the limit case is two-fold. First, when $N$ goes to infinity, the empirical distribution of the minor players' states converges to a random measure $\mu_t$ which we expect to be the conditional distribution of any minor player's state, i.e.: 
\[
\mu_t^N \rightarrow \mu_t := \bigl(\mathbb{P}[X_t^n = 1 | X_t^0], \mathbb{P}[X_t^n = 2 | X_t^0], \dots, \mathbb{P}[X_t^n = M-1 | X_t^0]\bigr).
\]
As we shall see later on in the next section, when considered together with the major player's state and one of the minor player's state, the resulting process is Markovian and its infinitesimal generator has a tractable form. Also, when the number of minor players goes to infinity, small perturbations of a single minor player's strategy will have no significant influence on the distribution of minor player's states. this gives rise to a simple formulation of the typical minor player's search for the best response to the control choices of the major player. In the limit $N\to\infty$, we understand a Nash equilibrium as a situation in which neither the major player, nor a typical minor player could be better off by changing control strategy. In order to formulate this limiting problem, we need to define the joint dynamics of the states of the major player and a representative minor player, making sure that the dynamics of the state of the major player depend upon the statistical distribution of the states of the minor players, and that the dynamics of the state of the representative minor player depend upon the values of the state and the control of the major player, its own state, and the statistical distribution of the states of all the minor players.

As argued in \cite{CarmonaWang_LQ}, and echoed in Chapter 13 of \cite{CarmonaDelarue_book_II}, the best way to search for
Nash equilibria in the mean field limit of games with major and minor players is first to identify the best response map 
of the major and a representative of the minor players by solving the optimization problems for the strategies of 1) the major player in response to the field of minor players as represented by a special minor player with special state dynamics which we call a representative minor player, and 2)  the representative minor player in response to the behavior of the major player and the other minor players. Solving these optimization problems separately provides a definition of the best response map for the system. One can then search for a fixed point for this best response map. So the search for Nash equilibria for the mean field game with major and minor players can be summarized in the following two steps.

\vskip 6pt\noindent
\textbf{Step 1} (Identifying the Best Response Map) 

\vskip 6pt\noindent
\textbf{1.1} (Major Player's Problem) 
\vskip 1pt\noindent
Fix an admissible strategy $\balpha\in\mathbb{A}$
of the form $\alpha_t = \phi(t, X_t, X^0_t, \mu_t)$ for the representative minor player, solve for the optimal control problem of the major player given that all the minor players use the feedback function $\phi$. We denote by $\bphi^{0,*}(\phi)$ the feedback function giving the optimal strategy of this optimization problem.

Notice that, in order to formulate properly this optimization problem, we need to define Markov dynamics for the couple $(X^0_t,X_t)$ where $X_t$ is interpreted as the state of a representative minor player, and the (random) measure $\mu_t$ has to be defined clearly. this is done in the next section as the solution of the major player optimization problem, the Markovian dynamics being obtained from the limit of games with $N$ minor players. 

\vskip 6pt\noindent
\textbf{1.2}  (Representative Minor Player's Problem) 
\vskip 1pt\noindent
We first single out a minor player and we search for its best response to the rest of the other players. So we
fix an admissible strategy $\balpha^0\in \mathbb{A}^0$ of the form $\alpha^0_t= \phi^0(t, X_t^0, \mu_t)$ for the major player,
and an admissible strategy $\balpha\in \AA$ of the form $\alpha_t= \phi(t, X_t, X_t^0, \mu_t)$ for the representative of the remaining minor players.
We then assume that the minor player which we singled out responds to the other players by choosing 
an admissible strategy $\bar\balpha\in \AA$ of the form $\bar\alpha_t= \bar\phi(t,\o X_t, X_t^0, \mu_t)$. 
Clearly, if we want to find the best response of the singled out minor player to the behavior of the major player and the field of the other minor players as captured by the behavior of the representative minor player, we need to define Markov dynamics for the triple $(X^0_t,\o X_t,X_t)$, and define clearly what we mean by the (random) measure $\mu_t$. this is done in the next section as the solution of the representative minor player optimization problem, the Markovian dynamics being obtained from the limit of games with $N$ minor players.  
We denote by $\bphi^{*}(\phi^0,\phi)$ the feedback function giving the optimal strategy of this optimization problem.

\vskip 6pt\noindent
\textbf{Step 2} (Search for a Fixed Point of the Best Response Map)
\vskip 1pt\noindent
A Nash equilibrium for the mean field game with major and minor players is a fixed point $[\hat\phi^{0},\hat\phi] = [\bphi^{0,*}(\hat\phi), \bphi^*(\hat\phi^0, \hat\phi)]$.

\vskip 4pt
Clearly, in order to take Step 1, we need to formulate properly the search for these two best responses, and study the limit $N\to\infty$ of both cases of interest.

\section{Convergence of Large Finite Player Games}
\label{se:convergence}

Throughout the rest of the paper, we make the following assumptions on the regularity of the transition rate and cost functions:

\begin{hypothesis}\label{lipassump}
There exists a constant $L>0$ such that for all $i,j\in\{1,\dots, M\}$, $i^0, j^0 \in \{1, \dots, M^0\}$ and all $t,t'\in[0,T]$, $x,x'\in \mathcal{P}$, $\alpha^0, \alpha^{0 '} \in A_0$ and $\alpha, \alpha'\times A$, we have:
\begin{align*}
&|(f,f^0,g,g^0, q, q^0)(i,j,i^0, j^0, t,x,\alpha^0, \alpha) - (f,f^0,g,g^0, q, q^0)(i,j,i^0, j^0, t',x',\alpha^{0'}, \alpha')|\\
&\hskip 35pt
\le L (|t-t'| + \|x-x'\| + \| \alpha^0 - \alpha^{0'}\| + \|\alpha -\alpha'\|)
\end{align*}
\end{hypothesis}

\begin{hypothesis}\label{boundassump}
There exists a constant $C>0$ such that for all $i,j\in\{1,\dots, M\}$, $i^0\in \{1, \dots, M^0\}$ and all $t\in[0,T]$, $x\in \mathcal{P}$, $\alpha^0\in A_0$ and $\alpha\in A$, we have:
\[
|q(t,i,j,\alpha,i^0, \alpha^0, x)|\le C
\]
\end{hypothesis}

Finally, we add a boundary condition on the Markov evolution of the minor players. Intuitively speaking, this assumption rules out extinction: it says that a minor player can no longer change its state, when the percentage of minor players who are in the same state falls below a certain threshold.
\begin{hypothesis}
\label{boundaryassump}
There exists a constant $\epsilon>0$ such that for all $t\in[0,T]$, $i,j \in \{1,\dots, M-1\}, i\neq j$ and $\alpha^0\in A_0$ and $\alpha\in A$, we have:
\begin{align*}
x_i < \epsilon \implies&\;\; q(t,i,j,\alpha, i^0, \alpha^0, x) = 0\\
1 - \sum_{k=1}^{M-1} x_k < \epsilon \implies&\;\; q(t,M,i,\alpha, i^0, \alpha^0, x) = 0.
\end{align*}
\end{hypothesis}
\noindent
The purpose of this section is to identify the state dynamics which should be posited in the formulation of the mean field game problem with major and minor players.  In order to do so, we formulate the search for the best response of each player by first setting the game with finitely many minor players, and then letting the number of minor players go to $\infty$ to identify the dynamics over which the best response should be computed in the limit.

\subsection{Major Player's Problem with Finitely Many Minor Players}

For any integer $N$ (fixed for the moment), we consider a game with $N$ minor players, and we compute the best response of the major player when the minor players choose control strategies $\balpha^n=(\alpha^n_t)_{0\le t\le T}$ given by the same 
feedback function $\phi$ so that $\alpha^n_t = \phi(t, X_t^{n,N}, X^{0,N}_t, \mu^N_t)$ for $n=1,\cdots,N$. 
Here $X_t^{n,N}$ denotes the state of the $n$-th minor player at time $t$,  $X^{0,N}_t$ the state of the major player, and $\mu^N_t$ the empirical distribution of the states of the $N$ minor players at time $t$. The latter is a probability measure on the state space 
$E=\{1,\cdots,M\}$, and for the sake of convenience, we shall identify it with the element:
\[
\mu_t^N = \frac{1}{N}\bigg(\sum_{n=1}^N \mathbbm{1}(X_t^{n,N} = 1), \sum_{n=1}^N \mathbbm{1}(X_t^{n,N} = 2), \dots, \sum_{n=1}^N \mathbbm{1}(X_t^{n,N} = M-1)\bigg)
\]
of the simplex. So for each $i\in E$, $\mu_t^N(i)$ is the proportion of minor players whose state at time $t$ is equal to $i$. Consequently, for $N$ fixed,
$\mu^N_t$ can be viewed as an element of the finite space $\{0,1/N,\cdots,(N-1)/N,1\}^{M-1}$. For the sake of definiteness, we denote by 
$\cP^N$ the set of possible values of $\mu^N_t$, in other words, we set:
\[
\mathcal{P}^N:=\bigl\{\frac{1}{N}(n_1, n_2, \dots n_{M-1}) ;\; n_i \in \mathbb{N}, \sum_i n_i \le N\bigr\}.
\]
Given the choice of control strategies made by the minor players, we denote by $\balpha^{0}=(\alpha^0_t)_{0\le t\le T}$ the control strategy of the major player, and we study the time evolution of the state of the system given these choices of control strategies. Later on, we shall find the optimal choice for major player's controls $\balpha^{0}$ given by feedback functions $\phi^0$ in response to the choice of the feedback function $\phi$ of the minor players. While this optimization should be done over the dynamics of the whole state $(X^{0,N}_t,X^{1,N}_t,\cdots,X^{N,N}_t)$, we notice that the process 
$(X^{0,N}_t,\mu^N_t)_{0\le t\le T}$ is sufficient to define the optimization problem of the major player, and that it is also a continuous time Markov process in the finite state space $\{1,\dots,M^0\}\times\mathcal{P}^N$.

\vskip 4pt
Our goal is to show that as $N\to\infty$, the Markov process $(X_t^{0,N}, \mu_t^N)_{0\le t\le T}$ converges in some sense to a Markov process $(X_t^0, \mu_t)_{0\le t\le T}$. This will allow us to formulate the optimization problem of the major player in the mean field limit in terms of this limiting Markov process.

\vskip 4pt
For each integer $N$, we denote by $\cG^{0,N}_{\phi^0, \phi}$ the infinitesimal generator of the Markov process $(X_t^{0,N}, \mu_t^N)_{0\le t\le T}$. Since the process is not time homogeneous, when we say infinitesimal generator, we mean the infinitesimal generator of the space-time process $(t,X_t^{0,N}, \mu_t^N)_{0\le t\le T}$. Except for the partial derivative with respect to time, this infinitesimal generator is given by the Q-matrix of the process, namely the instantaneous rates of jump in the state space $\{1,\dots,M^0\}\times\mathcal{P}^N$. So if 
$F: [0,T] \times \{1, 2, \dots, M^0\} \times \mathcal{P}^N\rightarrow \mathbb{R}$ is $C^1$ in time, 
\begin{equation}
\label{generatorNmajor}
\begin{aligned}
[\cG^{0,N}_{\phi^0, \phi} F] (t, i^0, x)=&\;\partial_t F(t, i^0, x)+ \sum_{j_0 \neq i_0} \bigl(F(t, j^0, x) - F(t, i^0, x)\bigr)   q^0_{\phi^0}(t, i^0, j^0, x)\\
&\; + \sum_{j\neq i}\bigl(F(t, i^0, x + \frac{1}{N}e_{ij} ) -F(t, i^0, x )\bigr) N x_i  q_{\phi^0,\phi}(t, i, j, i^0, x),
\end{aligned}
\end{equation}
where the first summation in the right hand side corresponds to jumps in the state of the major player and the terms in the second summation account for the jumps of the state of one minor player from $i$ to $j$. Here we code the change in the empirical distribution $x$ of the states of the minor players 
caused by the jump from $i\in\{1,\cdots,M\}$ to $j\in\{1,\cdots,M\}$ with $j\ne i$, of the state of a single minor player as $(1/N)e_{ij}$ with the notation $e_{ij} :=  e_j \mathbbm{1}_{j\neq M} - e_i \mathbbm{1}_{i\neq M}$ where $e_i$ stands for the $i$-th vector in the canonical basis of the space $\mathbb{R}^{M-1}$. We have also used the notation $x_n = 1 - \sum_{i=1}^{n-1} x_i$ for sake of simplicity.

\vskip 4pt
Notice that the two summations appearing in \eqref{generatorNmajor} correspond to finite difference operators which are bounded. So the domain of the operator $\cG^{0,N}_{\phi^0, \phi} $ is nothing else than the domain of the partial derivative with respect to time.
Notice also that the sequence of generators ${\cG}^{0,N}_{\phi^0, \phi}$ converges, at least formally, toward a limit which can easily be identified. Indeed, it is clear from the definition \eqref{generatorNmajor} that $[\cG^{0,N}_{\phi^0, \phi} F] (t, i^0, x)$ still makes sense if $x\in\cP$, where $\cP$ is the $M-1$ dimensional simplex. 
Moreover, if $F: [0,T] \times \{1, 2, \dots, M^0\} \times \mathcal{P}\rightarrow \mathbb{R}$ is $C^1$ in both variables $t$ and $x$, we have $[\cG^{0,N}_{\phi^0,  \phi}  F] (t,i^0,x) \rightarrow [\cG^{0}_{\phi^0,  \phi}  F](t,i^0,x)$ defined by:
\begin{align*}
[\cG^{0}_{\phi^0,  \phi}  F](t,i^0,x)&:= \partial_t F(t, i^0, x) + \sum_{j^0\neq i^0}[F(t, j^0, x) - F(t, i^0, x)]  q^0_{\phi^0}(t,i^0, j^0, x)\\
&+ \sum_{i,j=1}^{M-1} \partial_{x_j} F(t, i^0, x)  x_i   q_{\phi^0,\phi}(t,i,j,i^0, x) +(1- \sum_{k=1}^{M-1} x_k)   \sum_{j=1}^{M-1} \partial_{x_j} F(t, i^0, x) q_{\phi^0,\phi}(t,M, j, i^0, x).
\end{align*}
So far, we have a sequence of time-inhomogeneous Markov processes $(X_t^{0, N}, \mu_t^N)$ characterized by their infinitesimal generators $\mathcal{G}^{0,N}_{\phi^0, \phi}$  which converge to $\mathcal{G}^{0}_{\phi^0, \phi}$. We now aim to show the existence of a limiting Markov process with infinitesimal generator $\mathcal{G}^{0}_{\phi^0, \phi}$. The proof consists of first showing the existence of a Feller semigroup generated by the limiting generator $\mathcal{G}^{0}_{\phi^0, \phi}$, and then applying an argument of convergence of semigroups. 
\begin{remark}
The standard results in the theory of semigroup are tailor-made for time-homogeneous Markov processes. However, they can easily be adapted to the case of time-inhomogeneous Markov process by simply considering the space-time expansion, specifically by augmenting the process $(X_t^{0, N},\mu_t^N)$ into  $(t, X_t^{0, N}, \mu_t^N)$ and considering the uniform convergence on all bounded time intervals.
\end{remark}

\vskip 4pt
Let us introduce some notations which are useful for the functional analysis of the infinitesimal generators and their corresponding semigroups. We set 
$E^N=[0,T]\times \{1,\dots,M^0\}\times \mathcal{P}^N$ and $E^\infty=[0,T]\times \{1,\dots,M^0\}\times \mathcal{P}$  for the state spaces, and we denote by $C(E^\infty)$ the Banach space for  the norm $\|F\|_\infty=\sup_{t,i^0,x} |F(t,i^0,x)|$, of the real valued continuous functions defined on $E^\infty$. We also denote by $C^1(E^\infty)$ the collection of functions in $C(E^\infty)$ that are $C^1$ in $t$ and $x$ for all $i^0 \in \{1,\dots,M^0\}$.

Note that the Markov process $(t, X_t^{0,N},\mu_t^N)$ lives in $E^N$ while the candidate limiting process $(t, X_t^{0,N},\mu_t^N)$ lives in $E^\infty$. The difference is that $\mu_t^N$ only takes values in $\mathcal{P}^N$, which is a finite subset of $\mathcal{P}$. Thus if we want to show the convergence, we need to reset all the processes on the same state space, and our first step should be to extend the definition of $(t,X_t^{0,N},\mu_t^N)$ to a Markov process taking value in $E^\infty$. To do so, we extend the definition of the generator $\cG^{0,N}_{\phi^0,\phi}$ to accommodate functions $F$ defined on the whole $E^\infty$:
\begin{align*}
[\mathcal{G}^{0,N}_{\phi^0,\phi} F] (t, i^0, x)=\;&\partial_t F(t, i^0, x)+ \sum_{j_0 \neq i_0} \bigl(F(t, j^0, x) - F(t, i^0, x)\bigr)  q^0_{\phi^0}(t, i^0, j^0, x)\\
&\; + \sum_{j\neq i}\bigl(F(t, i^0, x + \frac{1}{N}e_{ij} ) -F(t, i^0, x )\bigr) N x_i  \mathbbm{1}_{x_i \ge \frac{1}{N}}  q_{\phi^0,\phi}(t, i, j, i^0, x).
\end{align*}
We claim that for $N$ large enough, $\mathcal{G}^{0,N}_{\phi^0,\phi}$ generates a Markov process with a Feller semigroup taking values in $E^\infty$. Indeed, when the initial distribution is a probability measure on $\{1,\dots, M^0\}\times\mathcal{P}^N$, the process has exactly the same law as $(X_t^{0,N}, \mu_t^N)$. To see why this is true, let us denote for all $x\in\mathcal{P}$ the set $\mathcal{P}_{x}^N:=(x + \frac{1}{N}\mathbb{Z}^{M-1})\cap \mathcal{P}$. Then we can construct a Markov process starting from $(i,x)$ and living in the space of finite states $\{1,\dots,M\}\times \mathcal{P}_{x}^N$, which has the same transition rates as those appearing in the definition of $\mathcal{G}^{0,N}_{\phi^0,\phi}$. In particular, the indicator function $\mathbbm{1}_{x_i \ge \frac{1}{N}}$ forbids the component $x$ to exit the domain $\mathcal{P}$. Hypothesis \ref{boundaryassump} implies that the transition function is continuous on $E$ when $N \ge 1/\epsilon$, where $\epsilon$ is the extinction threshold in the assumption. So this process is a continuous time Markov process with continuous probability kernel in a compact space. By Proposition 4.4 in \cite{swart}, it is a Feller process. In the following, we will still denote this extended version of the process as $(X_t^{0,N}, \mu_t^N)$.

\begin{proposition}\label{existfeller}
There exists a Feller semigroup $\cT=(\cT_t)_{t\ge 0}$ on the space $C(E^\infty)$ such that the closure of $\mathcal{G}^{0}_{\phi^0, \phi}$ is the infinitesimal generator of $\mathcal{T}$.
\end{proposition}

\begin{proof}
We use a simple perturbation argument. Observe that $\mathcal{G}^{0}_{\phi^0, \phi}$ is the sum of two linear operator $\mathcal{H}$ and $\mathcal{K}$ on $C(E^{\infty})$:
\begin{align*}
[\mathcal{H}F](t,i^0,x) :=& \partial_t F(t,i^0,x) + \mathbf{v}(t,i^0,x) \cdot \nabla F(t,i^0,x)\\
[\mathcal{K}F](t,i^0,x) :=& \sum_{j^0\neq i^0}[F(t, j^0, x) - F(t, i^0, x)]  q^0_{\phi^0}(t,i^0, j^0, x)
\end{align*}
where we denote by $\nabla F(t,i^0,x)$ the gradient of $F$ with respect to $x$, and by $\mathbf{v}$ the vector field:
\[
\mathbf{v}_j(t,i^0,x) := \sum_{i=1}^{M-1} x_i q_{\phi^0,\phi}(t,i,j,i^0,x) + \Bigl(1 - \sum_{i=1}^{M-1} x_i \Bigr) \;q_{\phi^0,\phi}(t,M,j,i^0,x).
\]
Being a finite difference operator, $\mathcal{K}$ is a bounded operator on $C(E^{\infty})$ so the proof reduces to showing that $\cH$ generates a Feller semigroup. See for example Theorem 7.1, Chapter 1 in \cite{EthierKurtz}.
To show that the closure of $\mathcal{H}$ generates a strongly continuous semigroup on $C(E^\infty)$ we use the characteristics of the vector field $\mathbf{v}$. For any $(t,i^0,x) \in E^{\infty}$, let $(Y^{t,i^0,x}_u)_{u\ge 0}$ be the solution of the Ordinary Differential Equation (ODE):
\[
dY^{t,i^0,x}_u = \mathbf{v}(t+u,i^0,Y^{t,i^0,x}_u) du,\;\;\;\;Y^{t,i^0,x}_0 = x.
\]
Existence and uniqueness of solutions are guaranteed by the Lipschitz property of the vector field $\mathbf{v}$, which in turn is a consequence of the Lipschitz property of $q_{\phi^0, \phi}$. Notice that by Hypothesis \ref{boundaryassump}, the process $Y^{t,i^0,x}_u$ is confined to $\mathcal{P}$. So we can define the linear operator $\mathcal{T}_s$ on $C(E^{\infty})$:
\[
[\mathcal{T}_s F](t,i^0,x) := F(s+t, i^0, Y^{t,i^0,x}_s)
\]
Uniqueness of solutions implies that $(\mathcal{T}_s)_{s\ge 0}$ is a semigroup. The latter is strongly continuous. Indeed, by the boundedness of $\mathbf{v}$, for a fixed $h>0$, there exists a constant $C_0$ such that $|Y^{t,i^0,x}_s - x| \le C_0 s$ for all $s\le h$ and $(t,i^0,x) \in E^{\infty}$. Combining this estimation with the fact that $F$ is uniformly continuous in $(t,x)$ for all $F\in C(E^{\infty})$, we obtain that $\|\mathcal{T}_s F - F\| \rightarrow 0, s\rightarrow 0$. Finally the semigroup $\mathcal{T}$ is Feller, since the solution of ODE $Y^{t,i^0,x}_s$ depends continuously on the initial condition as a consequence of the Lipschitz property of the vector field $\mathbf{v}$.

It is plain to check that $\mathcal{H}$ is the infinitesimal generator of the semigroup $\mathcal{T}$, and the domain of $\mathcal{H}$ is $C^1(E^{\infty})$. \end{proof}

The following lemma is a simple adaptation of Theorem 6.1, Chapter 1 in \cite{EthierKurtz} and is an important ingredient in the proof of the convergence. It says that the convergence of the infinitesimal generators implies the convergence of the corresponding semigroups.
\begin{lemma}\label{convgeneratorsemigroup}
For $N=1,2,\dots $, let $\{T_N(t)\}$ and $\{T(t)\}$ be strongly continuous contraction semigroups on $L$ with generator $\mathcal{G}_N$ and $\mathcal{G}$ respectively. Let $D$ be a core for $\mathcal{G}$ and assume that $D \subset \mathcal{D}(\mathcal{G}_N)$ for all $N\ge 1$. If $\lim_{N\to\infty} \mathcal{G}_N F = \mathcal{G}  F$ for all $F \in D$, then for each $F\in L$, $\lim_{N\to\infty} T_N(t) F = T(t) F$ for all $t\ge 0$.
\end{lemma}

We are now ready to state and prove the main result of this section: the Markov process $(X_t^{0,N}, \mu_t^N)$ describing the dynamics of the state of the major player and the empirical distribution of the states of the $N$ minor players converges  weakly to a Markov process with infinitesimal  generator $\mathcal{G}^{0}_{\phi^0, \phi}$, when the players choose Lipschitz strategies. 

\begin{theorem}\label{convtheomajor}
Assume that the major player chooses a control strategy $\balpha^0$ given by a Lipschitz feedback function $\phi^0$ and that all the minor players choose control strategies given by the same Lipschitz feedback function $\phi$. Let $i^0\in \{1,\dots, M^0\}$ and for each integer $N\ge 1$, $x^N \in \mathcal{P}^N$ 
with limit $x \in \mathcal{P}$. Then the sequence of processes $(X_t^{0,N}, \mu_t^N)$ with initial conditions $X_0^{0,N} = i^0$, $\mu_t^N = x^N$ converges weakly to a Markov process $(X_t^0, \mu_t)$ with initial condition $X_0^0 = i^0$, $\mu_0 = x$. The infinitesimal generator for $(X_t^0, \mu_t)$ is given by:
\begin{align*}
[\mathcal{G}^0_{\phi^0, \phi} F] (t, i^0, x):=&\;\partial_t F(t, i^0, x) + \sum_{j^0\neq i^0}[F(t, j^0, x) - F(t, i^0, x)] q^0_{\phi^0}(t,i^0, j^0, x)\\
& + \sum_{i,j=1}^{M-1} \partial_{x_j} F(t, i^0, x)  x_i   q_{\phi^0,\phi}(t,i,j, i^0, x)+(1- \sum_{k=1}^{M-1} x_k)   \sum_{k=1}^{M-1} \partial_{x_j} F(t, i^0, x) q_{\phi^0,\phi}(t,M, j, i^0, x).
\end{align*}
\end{theorem}

\begin{proof}
Let us denote $\mathcal{T}^N$ the semigroup associated with the time inhomogeneous Markov process $(t, X_t^{0,N},\mu_t^N)$ and the infinitesimal generator $\mathcal{G}^{0,N}_{\phi^0,\phi}$. Recall that by the procedure of extension we described above, the process $(t, X_t^{0,N},\mu_t^N)$ now lives in $E^{\infty}$ and the domain for $\mathcal{G}^{0,N}_{\phi^0,\phi}$ is $C(E^{\infty})$. In light of Theorem 2.5, Chapter 4 in \cite{EthierKurtz} and Proposition \ref{existfeller} we just proved, it boils down to proving that for any $F \in E^{\infty}$ and $t\ge 0$, $\mathcal{T}_t^N F$ converges to $\mathcal{T}_t F$, where $(\mathcal{T}_t)_{t\ge 0}$ is the strongly continuous semigroup generated by the closure of $\mathcal{G}_{\phi^0,\phi}^0$.

To show the convergence, we apply Lemma \ref{convgeneratorsemigroup}. It is easy to see that $C^1(E^\infty)$ is a core for $\bar{\mathcal{G}}_{\phi^0,\phi}^0$ and $C^1(E^\infty)$ is included in the domain of $\mathcal{G}^{0,N}_{\phi^0,\phi}$. Therefore it only remains to show that for all $F\in C^1(E^\infty)$, $\mathcal{G}^{0,N}_{\phi^0,\phi} F$ converges to $\mathcal{G}^{0}_{\phi^0,\phi} F$ in the space $(C(E^\infty), \|\cdot\|)$. Using the notation $x_M := 1 - \sum_{i=1}^{M-1} x_i$, we have:
\begin{equation}\label{generatorapproxest}
\begin{aligned}
&|[\mathcal{G}_{\phi^0,\phi}^{0,N} F](t,i^0,x) - [\mathcal{G}^{0}_{\phi^0,\phi} F](t,i^0,x)|\\
= & \sum_{j\neq i}|N(F(t, i^0, x + \frac{1}{N}e_{ij} ) -F(t, i^0, x )) - (\mathbbm{1}_{j\neq M} \partial_{x_j} F(t, i^0, x ) - \mathbbm{1}_{i\neq M} \partial_{x_i}F(t, i^0, x ))| x_i  q^N(t, i, j, i^0, x)\\
\le & \sum_{j\neq i}( |\partial_{x_j} F(t, i^0, x + \frac{\lambda_{i,j}}{N}e_{i,j} ) -\partial_{x_j} F(t, i^0, x )|+ |\partial_{x_i}F(t, i^0, x + \frac{\lambda_{i,j}}{N}e_{i,j}  ) - \partial_{x_i}F(t, i^0, x )|) x_i  q^N(t, i, j, i^0, x)
\end{aligned}
\end{equation}
where we applied intermediate value theorem at the last inequality and $\lambda_{i,j}\in[0,1]$. Note that $\lambda_{i,j}$ also depends on $t,x,i^0$ but we omit them for sake of the simplicity. Remark that $F\in C^1(E^\infty)$ and $E^\infty$ is compact, therefore $\partial_{x_i}F$ is uniformly continuous on $E^\infty$ for all $i$, which immediately implies that $\|\mathcal{G}_{\phi^0,\phi}^{0,N} F - \mathcal{G}_{\phi^0,\phi}^{0} F\| \rightarrow 0, N\rightarrow +\infty$. This completes the proof.
\end{proof}

\noindent
\textbf{Mean Field Major Player's Optimization Problem}

\noindent Given that all the minor players are assumed to use a control strategy based on the same feedback function $\phi$, the best response of the major player is to use the strategy $\hat\balpha^0$ given by the feedback function $\hat\phi^0$ solving the optimal control problem:
\[
\inf_{\balpha^0\leftrightarrow\phi^0 \in\mathbb{A}^0}\mathbb{E}\left[\int_{0}^{T}f^0(t, X_t^0, \phi^0(t, X_t^0, \mu_t), \mu_t) dt + g^0(X_T^0, \mu_T)\right]
\]
where $(X_t^0, \mu_t)_{0\le t\le T}$ is the continuous time Markov process with infinitesimal generator $\mathcal{G}^0_{\phi^0,\phi}$. 

\subsection{Representative Minor Player's Problem}
We turn to the computation of the best response of a generic minor player. We assume that the major player chooses a strategy $\balpha^0\in \mathbb{A}^0$ of the form $\alpha^0_t= \phi^0(t, X_t^{0,N}, \mu^N_t)$ and that the minor players in $\{2,\cdots,N\}$ all use strategy  $\balpha^i\in \mathbb{A}$ of the form $\alpha^i_t= \phi(t, X^{i,N}_t,X_t^{0,N}, \mu^N_t)$ for $i=2,\cdots,N$, and that the first minor player uses strategy  $\o\balpha\in \mathbb{A}$ of the form $\o\alpha_t= \o\phi(t, X^{1,N}_t,X_t^{0,N}, \mu^N_t)$. Clearly, by symmetry, whatever we are about to say after we singled the first minor player out, can be done if we single out any other minor player. As before, for each fixed $N$, the process $(X_t^{0,N}, X_t^{1,N}, \mu_t^N)$ is a finite-state continuous time Markov process with state space $\{1,\dots, M^0\}\times \{1,\dots, M\}\times\mathcal{P}^N$ whose infinitesimal generator $\mathcal{G}^{N}_{\phi^0, \phi, \o\phi}$ is given, up to the time derivative, by the corresponding Q-matrix of infinitesimal jump rates. In the present situation, its value on any real valued function $F$ defined on $[0,T]\times \{1,\dots, M^0\}\times \{1,\dots, M\}\times\mathcal{P}^N$ such that $t\rightarrow F(t,i^0,i,x)$ is $\mathcal{C}^1$ for any $i^0,i$ and $x$ is given by the formula:
\begin{equation}\label{generatorNminor}
\begin{aligned}
[\mathcal{G}^{N}_{\phi^0, \phi,\o\phi}  F] (t, i^0, i, x) =& \partial_t F(t, i^0, i, x)+ \sum_{j^0, j_0 \neq i_0} [F(t, j^0,i, x) - F(t, i^0,i, x)]  q^0_{\phi^0}(t, i^0, j^0, x)\\
&+ \sum_{j, j\neq i}[F(t, i^0, j, x + \frac{1}{N}e_{ij} ) -F(t, i^0,i, x )] q_{\phi^0, \o\phi}(t, i, j, i^0, x)\\
&+ \sum_{(j,k), j\neq k}[F(t, i^0, i, x +\frac{1}{N} e_{kj} ) - F(t, i^0, i, x)] (N x_k -  \mathbbm{1}_{k=i}) q_{\phi^0,\phi}(t, k, j, i^0, x).
\end{aligned}
\end{equation}
As before the summations appearing above correspond to single jumps when 1) only the state of the major player changes from state $i^0$ to $j^0$, 2) only the state of the singled out first minor player changes from state $i$ to $j$, and finally 3) the state of one of the last $N-1$ minor players jumps from state $k$ to $j$. 

\vskip 4pt
Following the same treatment as in major player's problem, we have the convergence result for the process $(X_t^{N}, X_t^{0,N},\mu_t^N)$:

\begin{theorem}\label{convtheominor}
Assume that for each integer $N$, the major player chooses a control $\balpha^0$ given by a Lipschitz feedback function $\phi^0$, the first minor player chooses a control $\o\balpha$ given by a Lipschitz feedback function $\o\phi$, all the other minor players choose strategies given by the same Lipschitz feedback function $\phi$, and that these three feedback functions do not depend upon $N$. Let $i^0\in \{1,\dots, M^0\}$ and for each integer $N\ge 2$, 
$x^N \in \mathcal{P}^N$ with limit $x \in \mathcal{P}$. Then the sequence of processes $(X_t^{N}, X_t^{0,N}, \mu_t^N)_{0\le t\le T}$ with initial conditions $X_0^N = i$, $X_0^{0,N}= i^0$ and $\mu_0^N = x^N$ converges weakly to a Markov process $(X_t, X_t^0, \mu_t)$ with initial condition $X_0 = i$, $X_0^0 = i^0$ and $\mu_0 = x$. Its infinitesimal generator is given by:
\begin{align*}
[\mathcal{G}_{\phi^0, \phi, \o\phi}  F] (t, i, i^0, x) :=& \;\partial_t F(t,i, i^0, x) + \sum_{j^0, j^0\neq i^0}[F(t,i, j^0, x) - F(t,i, i^0, x)] q^0_{\phi^0}(t, i^0, j^0, x) \\
&\hskip -35pt
 + \sum_{j, j\neq i}[F(t,j, i^0, x) - F(t,i, i^0, x)] q_{\phi^0,\o\phi}(t,i, j, i^0, x)+ \sum_{i,j=1}^{M-1} \partial_{x_j} F(t, i, i^0, x)  x_i   q_{\phi^0,\phi}(t,i,j,i^0,x)\\
&\;+(1- \sum_{k=1}^{M-1} x_k)   \sum_{j=1}^{M-1} \partial_{x_k} F(t,i, i^0, x) q_{\phi^0,\phi}(t,M, j, i^0, x).
\end{align*}
\end{theorem}

\noindent \textbf{Representative Minor Player's Optimization Problem}

\noindent 
Accordingly, in the mean field game limit, we define the search for the best response of the representative minor player (i.e. the minor player we singled out) to the strategies adopted by the major player and the field of minor players as the following optimal control problem. Assuming that the major player uses a feedback function $\phi^0$ and all the other minor players the feedback function $\phi$, the best response of the representative minor player is given by the solution of:
\[
\inf_{\o\balpha\leftrightarrow \o\phi\in\mathbb{A}}\mathbb{E}\left[\int_{0}^{T} f(t,X_t, \bar\phi(t, X_t, X_t^0, \mu_t), X_t^0,\phi^0(t, X_t^0, \mu_t),  \mu_t) dt + g(X_T, X_T^0, \mu_T)\right]
\]
where $(X_t, X_t^0, \mu_t)_{0\le t\le T}$ is a  Markov process with infinitesimal generator $\mathcal{G}_{\phi^0, \phi, \o\phi}$. We shall denote by 
$\o\phi=\bphi(\phi^0,\phi)$ the optimal feedback function providing the solution of this optimal control problem.

\section{Optimization Problem for Individual Players}
\label{se:optimizations}
In this section, we use the dynamic programming principle to characterize the value functions of the major and minor players' optimization problems 
as viscosity solutions of the corresponding Hamilton-Jacobi-Bellman (HJB for short) equations.
We follow the detailed arguments given in Chapter II of \cite{FlemingSoner}. For both the major and representative minor player, we show that the value function solves a weakly coupled system of Partial Differential Equations (PDEs for short) in viscosity sense. We also prove an important uniqueness result 
for these solutions. This uniqueness result is important indeed because as the reader noticed, in defining the best response map, we implicitly assumed that these optimization problems could be solved and that their solutions were unique.

\vskip 4pt
We first consider the value function of the major player's optimization problem assuming that the minor players use the feedback function $\phi$:
\[
V^0_{\phi}(t,i^0, x) := \inf_{\balpha^0\leftrightarrow\phi^0}\mathbb{E}\left[\int_{t}^{T}f^0_{\phi^0}(s, X_s^0, \mu_s) ds + g^0(X_T^0, \mu_T) | X_t^0 = i^0, \mu_t = x\right]
\]

\begin{theorem}\label{hjbtheomajor}
Assume that for all $i^0 \in \{1,\dots, M^0\}$, the mapping $(t,x) \rightarrow V^{0}_{\phi}(t,i^0, x)$ is continuous on $[0,T]\times \mathcal{P}$. Then $V^0_{\phi}$ is a viscosity solution to the system of $M^0$ PDEs on $[0,T]\times \cP$:
\begin{equation}\label{hjbmajor}
\begin{aligned}
&0=\partial_t v^0(t, i^0, x) + \inf_{\alpha^0\in A^0}\bigg\{f^0(t, i^0, \alpha^0, x) + \sum_{j^0\neq i^0} [v^0(t,j^0, x) - v^0(t,i^0, x)] q^0(t,i^0, j^0, \alpha^0, x)\\
&\hskip 25pt
+(1-\sum_{k=1}^{M-1} x_k)  \sum_{k=1}^{M-1} \partial_{x_k} v^0(t,i^0,x)  q(t,M, k, \phi(t,M, i^0, x), i^0, \alpha^0, x)\\
&\hskip 25pt
+\sum_{i,j=1}^{M-1} \partial_{x_j} v^0(t, i^0, x)  x_i  q(t,i,j,\phi(t,i,i^0,x), i^0,\alpha^0, x)\bigg\},\qquad (i_0,t,x)\in \{1,\dots, M^0\} \times [0,T[ \times \mathcal{P},\\
&v^0(T,i^0, x) = g^0(i^0, x),\qquad (i^0, x)\in \{1,\dots, M^0\}\times\mathcal{P}.
\end{aligned}
\end{equation}
\end{theorem}

The notion of viscosity solution in the above result is specified by the following definition:

\begin{definition}\label{viscositydefi}
A real valued function $v^0$ defined on $[0,T]\times\{1,\dots,M^0\}\times\mathcal{P}$ such that $v^0(\cdot, i^0, \cdot)$ is continuous on $[0,T]\times\{1,\dots,M^0\}\times\mathcal{P}$ for all $i^0 \in \{1,\dots,M^0\}$ is said to be a viscosity subsolution (resp. supersolution) if for any $(t,i^0,x) \in [0,T]\times\{1,\dots,M^0\}\times\mathcal{P}$ and any $\mathcal{C}^\infty$ function $\theta$ defined on $[0,T]\times\mathcal{P}$ such that the function $(v^0(\cdot, i^0, \cdot) - \theta)$ attains a maximum (resp. minimum) at $(t,x)$ and $v^0(t,i^0,x) = \theta(t,x)$, the following inequalities holds:
\begin{align*}
0\le \text{(resp.}\ge)\;\;&\; \partial_t \theta(t, x) + \inf_{\alpha^0\in A^0}\bigg\{f^0(t, i^0, \alpha^0, x) + \sum_{j^0\neq i^0} [v^0(t,j^0, x) - v^0(t,i^0, x)] q^0(t,i^0, j^0, \alpha^0, x)\\
&\;+(1-\sum_{k=1}^{M-1} x_k)  \sum_{k=1}^{M-1} \partial_{x_k}\theta(t,x)  q(t,M, k, \phi(t,M, i^0, x), i^0,\alpha^0, x)\\
&\;+\sum_{i,j=1}^{M-1} \partial_{x_j} \theta(t, x)  x_i  q(t,i,j,\phi(t,i, i^0,x), i^0, \alpha^0,x)\bigg\},\;\text{if}\;\;\;t<T\\
0\le \text{(resp.}\ge)\;\;&g^0(i^0, x) - v^0(t,i^0,x),\;\;\;\text{if}\;\;\; t=T
\end{align*}
If $v^0$ is both a viscosity subsolution and supersolution, we call it a viscosity solution.
\end{definition}

\begin{proof}
Define $\mathcal{C}(\mathcal{P})^{M^0}$ the collection of functions $\theta(i^0, x)$ defined on $\{1,\dots,M^0\}\times\mathcal{P}$ such that $\theta(i^0,\cdot)$ is continuous on $\mathcal{P}$ for all $i^0$. Define the dynamic programming operator $\mathcal{T}_{t,s}$ on $\mathcal{C}(\mathcal{P})^{M^0}$ by:
\begin{equation}\label{dynamicprogoperatordef}
[\mathcal{T}_{t,s}\theta](i^0, x) := \inf_{\balpha^0\leftrightarrow\phi^0}\mathbb{E}\left[\int_{t}^{s} f^0_{\phi^0}(u, X_u^0, \mu_u) du + \theta(X_s^0, \mu_s) | X_t^0 = i^0, \mu_t = x\right]
\end{equation}
where the Markov process $(X_u^0, \mu_u)_{0\le t\le T}$ has infinitesimal generator $\mathcal{G}^0_{\phi^0,\phi}$. Then the value function can be expressed as:
\[
V^0_{\phi}(t,i^0,x) = [\mathcal{T}_{t,T} g^0](i^0, x)
\]
and the dynamic programming principle says that:
\[
V^0_{\phi}(t, i^0, x) = [\mathcal{T}_{t,s} V^0_{\phi}(s, \cdot, \cdot)](i^0, x), \qquad (t,s, i^0, x) \in[0,T]^2\times \{1,\dots,M_0\}\times \mathcal{P}.
\]
We will use the following lemma whose proof we give in the appendix.

\begin{lemma}\label{dynamicprogoperator}
Let $\Phi$ be a function on $[0,T]\times\{1,\dots, M^0\} \times \mathcal{P}$ and $i^0\in\{1,\dots, M^0\}$ such that $\Phi(\cdot, i^0, \cdot)$ is $\mathcal{C}^1$ in $[0,T]\times\mathcal{P}$ and $\Phi(\cdot, j^0, \cdot)$ is continuous in $[0,T]\times\mathcal{P}$ for all $j^0\neq i^0$. Then we have:
\begin{align*}
&\lim_{h \rightarrow 0}\frac{1}{h}\left[(\mathcal{T}_{t,t+h} \Phi(t+h, \cdot, \cdot))(i^0, x) - \Phi(t, i^0, x)\right] \\
&=\partial_t \Phi(t,i^0, x) + \inf_{\alpha^0\in A^0}\bigg\{f^0(t, i^0, \alpha^0, x) +(1-\sum_{k=1}^{M-1} x_k)  \sum_{k=1}^{M-1} \partial_{x_k} \Phi(t,i^0, x)  q(t,M, k, \phi(t,M, i^0, x), i^0, \alpha^0, x)\\
&\;+\sum_{i,j=1}^{M-1} \partial_{x_j} \Phi(t,i^0, x)  x_i  q(t,i,j,\phi(t,i, i^0,x), i^0,\alpha^0, x)+ \sum_{j^0\neq i^0} [\Phi(t,j^0, x) - \Phi(t,i^0, x)] q^0(t,i^0, j^0, \alpha^0, x)\bigg\}.
\end{align*}
\end{lemma}

We now prove the subsolution property. Let $\theta$ be a function defined on $[0,T]\times\mathcal{P}$ such that $(V^0(\cdot, i^0, \cdot) - \theta)$ attains maximum at $(t,x)$ and $V^0(t,i^0,x) = \theta(t,x)$. Define the function $\Phi$ on $[0,T]\times\{1,\dots, M^0\}\times\mathcal{P}$ by $\Phi(\cdot, i^0, \cdot) := \theta$ and $\Phi(\cdot, j^0, \cdot) := V^0(\cdot, j^0, \cdot)$ for $j^0\neq i^0$. Then clearly $\Phi \ge V^0$, which implies:
\[
(\mathcal{T}_{t,s} \Phi(s, \cdot, \cdot))(i^0, x) \ge (\mathcal{T}_{t,s} V^0(s, \cdot, \cdot))(i^0, x)
\]
By the dynamic programming principle and the fact that $\Phi(t,i^0, x) = V^0(t,i^0, x)$ we have:
\[
\lim_{h \rightarrow 0}\frac{1}{h}\left[(\mathcal{T}_{t,s} \Phi(s, \cdot, \cdot))(i^0, x) - \Phi(t, i^0, x)\right] \ge 0.
\]
Then applying the lemma we obtain the desired inequality. The viscosity property for supersolution can be checked in exactly the same way.
\end{proof}
For later reference, we state the comparison principle for the HJB equation we just derived. Again, its proof is postponed to the appendix.

\begin{theorem}\label{compmajor}(Comparison Principle)
Let us assume that the feedback function $\phi$ is Lipschitz, and let $w$ (resp. v) be a viscosity subsolution (resp. supersolution) of the equation (\ref{hjbmajor}). Then we have $w\le v$.
\end{theorem}

We now turn to the representative minor agent's optimization problem assuming that the major player uses the feedback function $\phi^0$ and all the other minor players use the feedback function $\phi$. We define the value function:
\[
V_{\phi^0,\phi}(t,i, i^0, x) := \inf_{\o\balpha\leftrightarrow \o\phi}\mathbb{E}\left[\int_{t}^{T} f_{\phi^0,\o\phi}(s, X_s,X^0_s,\mu_s) ds + g(X_T, X_T^0, \mu_T)  |X_t = i, X_t^0 = i^0, \mu_t = x\right]
\]
where the Markov process $(X_t,X^0_t,\mu_t)_{0\le t\le T}$ has infinitesimal generator $\cG_{\phi^0,\phi,\o\phi}$.
In line with the analysis of the major player's problem, we can show that $V_{\phi^0,\phi}$ is the unique viscosity solution to a coupled system of PDEs.

\begin{theorem}\label{hjbtheominor}
Assume that for all $i \in \{1,\dots, M\}$ and $i^0 \in \{1,\dots, M^0\}$, the mapping $(t,x) \rightarrow V_{\phi^0,\phi}(t,i,i^0, x)$ is continuous on $[0,T]\times \mathcal{P}$. Then $V_{\phi^0,\phi}$ is a viscosity solution to the system of PDEs:
\begin{equation}\label{hjbminor}
\begin{aligned}
&0=\partial_t v(t,i, i^0, x) \\
&\;\;+\inf_{\o\alpha\in A}\bigg\{f(t,i, \o\alpha, i^0, \phi^0(t, i^0, x), x) + \sum_{j\neq i} [v(t,j, i^0, x) - v(t,i, i^0, x)] q(t, i, j, \o\alpha, i^0, \phi^0(t, i^0, x), x)\bigg\}\\
&\;\;+ \sum_{j^0\neq i^0}[v(t,i, j^0, x) - v(t,i, i^0, x)] q^0_{\phi^0}(t, i^0, j^0, x) +(1-\sum_{k=1}^{M-1} x_k)  \sum_{k=1}^{M-1} \partial_{x_k} v(t,i,i^0,x)  q_{\phi^0,\phi}(t,M, k, i^0, x)\\
&\;\;+\sum_{i,j=1}^{M-1} \partial_{x_j} v(t,i, i^0, x)  x_i  q_{\phi^0,\phi}(t,i,j, i^0, x), \hskip 25pt (i,i_0,t,x)\in \{1,\dots, M\}\times\{1,\dots, M^0\} \times [0,T[ \times \mathcal{P}\\
&v(T,i,i^0, x) = g(i,i^0, x),\qquad (i,i_0,x)\in  \{1,\dots, M\}\times\{1,\dots, M^0\}\times\mathcal{P}.
\end{aligned}
\end{equation}
Moreover, if the feedback functions $\phi^0$ and $\phi$ are Lipschitz, then the above system of PDEs satisfies the comparison principle.
\end{theorem}
It turns out that the value functions $V^0_{\phi}$ and $V_{\phi^0,\phi}$ are Lipschitz in $(t,x)$. To establish this regularity property and estimate the Lipschitz constants, we need to first study the regularity of the value functions for the finite player games, and control the convergence in the regime of large games. We will state these results in Section \ref{se:chaos}, where we deal with the propagation of chaos and highlight more connections between finite player games and mean field games. 


We conclude this section with a result which we will use frequently in the sequel. To state it, we denote $J^0_{\phi^0,\phi}$ the expected cost of the major player when it uses the feedback function $\phi^0$ and  the minor players all use the feedback function $\phi$. Put differently:
\[
J^0_{\phi^0,\phi}(t,i^0,x) := \mathbb{E}\left[\int_{t}^{T}f^0_{\phi^0}(s, X_s^0, \mu_s) ds + g^0(X_T^0, \mu_T) | X_t^0 = i^0, \mu_t = x\right]
\]
where the Markov process $(X_s^0, \mu_s)_{t\le s\le T}$ has infinitesimal generator $\mathcal{G}^0_{\phi^0,\phi}$. Then by definition, we have $V_{\phi}^0(t,i^0, x) = \inf_{\balpha^0\leftrightarrow\phi^0} J^0_{\phi^0,\phi}(t, i^0, x)$. Similarly, we denote $J_{\phi^0,\phi,\o\phi}$ the expected cost of the representative minor player when it uses the feedback function $\o\phi$, while the major player uses the feedback function $\phi^0$ and all the other minor players use the same feedback function $\phi$:
\[
J_{\phi^0,\phi,\o\phi}(t,i,i^0,x) := \mathbb{E}\left[\int_{t}^{T}f_{\phi^0,\o\phi}(s, X_s, X_s^0,  \mu_s) ds + g(X_T, X_T^0, \mu_T) | X_t =i, X_t^0 = i^0, \mu_t = x\right]
\]
where the Markov process $(X_s^0,X^s, \mu_s)_{t\le s\le T}$ has infinitesimal generator $\mathcal{G}_{\phi^0,\phi,\o\phi}$.

\begin{proposition}\label{pdecharacterizepayoff}
If the feedback functions $\phi^0$, $\phi$ and $\o\phi$ are Lipschitz, then $J^0_{\phi^0,\phi}$ and $J_{\phi^0,\phi,\o\phi}$ are respectively continuous viscosity solutions of the PDEs (\ref{majorJpde}) and (\ref{minorJpde})
\begin{equation}
\label{majorJpde}
\begin{cases}
&0 =\;[\mathcal{G}^0_{\phi^0,\phi} v^0](t, i^0, x) + f^0_{\phi^0}(t, i^0, x)\\
&0 =\;g^0(i^0, x) - v^0(T,i^0, x),\qquad (i^0,x)\in\{1,\dots, M^0\}\times \mathcal{P}.
\end{cases}
\end{equation}

\begin{equation}
\label{minorJpde}
\begin{cases}
&0 =\; [\mathcal{G}_{\phi^0,\phi,\o\phi} v](t, i,i^0, x) + f_{\phi^0,\o\phi}(t, i, i^0,  x)\\
&0 =\;g(i,i^0, x) - v(T,i,i^0, x),\qquad (i,i_0,x)\in  \{1,\dots, M\}\times\{1,\dots, M^0\}\times\mathcal{P}.
\end{cases}
\end{equation}
Moreover, the PDEs (\ref{majorJpde}) and (\ref{minorJpde}) satisfy the comparison principle.
\end{proposition}
\begin{proof}
The continuity of $J^0_{\phi^0,\phi}$ and $J_{\phi^0,\phi,\o\phi}$ follows from the fact that $(X_t^0, \mu_t)$ and $(X_t, X_t^0,\mu_t)$ are Feller processes, which we have shown in Theorem \ref{convtheomajor} and Theorem \ref{convtheominor}. The viscosity property can be shown using the exact same technique as in the proof of Theorem \ref{hjbtheomajor}. Finally the comparison principle is a consequence of the Lipschitz property of $f^0,f,q^0,q,\phi^0,\phi,\o\phi$ and can be shown by slightly modifying the proof of Theorem \ref{compmajor}. We leave the details of the proof to the reader.
\end{proof}

\section{Existence of Nash Equilibria}
\label{se:Nash}

In this section, we prove existence of Nash equilibria when the minor player's jump rates and cost functions do not depend upon the major player's control. We work under the following assumption:

\begin{hypothesis}\label{assumpopt1}
Hypothesis \ref{lipassump} and Hypothesis \ref{boundassump} are in force. In addition, the transition rate function $q$ and the cost function $f$ for the minor player do not depend upon the major player's control $\alpha^0\in A_0$. 
\end{hypothesis}

The following assumptions will guarantee the existence of optimal strategies for both the major player and representative minor player.

\begin{hypothesis}\label{assumpopt2}
For all $i^0 = 1,\dots, M^0$, $(t,x) \in [0,T]\times\mathcal{P}$ and $v^0\in\mathbb{R}^{M^0}$, the function $\alpha^0 \rightarrow f^0(t,i^0,\alpha^0,x) + \sum_{j^0\neq i^0}(v^0_{j^0} - v^0_{i^0}) q^0(t,i^0, j^0, \alpha^0, x)$ has a unique maximizer in $A^0$ denoted as $\hat\alpha^0(t,i^0,x,v^0)$. Additionally, 
$\hat\alpha^0$ is Lipschitz in $(t,x,v^0)$ for all $i^0=1,\dots,M^0$ with common Lipschitz constant $L_{\alpha^0}$.
\end{hypothesis}

\begin{hypothesis}\label{assumpopt3}
For all $i = 1, \dots, M$, $i^0 = 1,\dots, M^0$, $(t,x) \in [0,T]\times\mathcal{P}$ and $v\in\mathbb{R}^{M \times M^0}$, the function $\alpha \rightarrow f(t,\alpha,i, i^0,x) + \sum_{j\neq i}(v_{j,i^0} - v_{i,i^0}) q(t,i, j, \alpha, i^0, x)$ has a unique maximizer in $A$ denoted as $\hat\alpha(t,i,i^0,x,v)$. Additionally, 
$\hat\alpha$ is Lipschitz in $(t,x,v)$  for all $i^0=1,\dots,M^0$ and $i=1,\dots,M$ with common Lipschitz constant $L_\alpha$.
\end{hypothesis}

\begin{proposition}\label{optimalresponseprop}
Under Hypothesis \ref{assumpopt1} - \ref{assumpopt3}, we have:

(i) For any Lipschitz feedback function $\phi$ for the representative minor player, the best response $\bphi^{0*}(\phi)$ of the major player exists and is given by:
\begin{equation}\label{optimalresponse1}
\bphi^{0*}(\phi)(t,i^0,x) = \hat\alpha^0(t, i^0, x, V^0_{\phi}(t,\cdot,x))
\end{equation}
where $\hat\alpha^0$ is the minimizer defined in Hypothesis \ref{assumpopt2} and $V^0_{\phi}$ is the value function of the major player's optimization problem.

(ii) For any Lipschitz feedback function $\phi^0$ for the major player and $\phi$ for the other minor players, the best response $\bphi^*(\phi^0, \phi)$ 
of the representative minor player exists and is given by:
\begin{equation}\label{optimalresponse2}
\bphi^*(\phi^0,\phi)(t,i,i^0,x) = \hat\alpha(t,i, i^0, x, V_{\phi^0,\phi}(t,\cdot,\cdot,x))
\end{equation}
where $\hat\alpha$ is the minimizer defined in Hypothesis \ref{assumpopt3} and $V_{\phi^0,\phi}$ is the value function of representative minor player's optimization problem.
\end{proposition}

\begin{proof}
Consider the expected total cost  $J^0_{\phi^{0*}(\phi),\phi}$ of the major player when all the minor players use the feedback function $\phi$ and the major player uses the strategy given by the feedback function $\phi^{0*}(\phi)$ defined by (\ref{optimalresponse1}). Also consider $V^0_{\phi}$ the value function of the major player's optimization problem. By definition of $\phi^{0*}(\phi)$ and the PDE (\ref{hjbmajor}), we see that $V^0_{\phi}$ is a viscosity solution of the PDE (\ref{majorJpde}) with $\phi^0 = \phi^{0*}(\phi)$ and $\phi= \phi$ in Proposition \ref{pdecharacterizepayoff}. To be able to use the comparison principle, we need to show that $\phi^0 = \phi^{0*}(\phi)$ and $\phi$ are Lipschitz. Indeed the Lipschitz property follows from Hypothesis \ref{assumpopt2} and Corollary \ref{valuefunctionlipschitzconstant} (see Section \ref{se:chaos}). Now since $J^0_{\phi^{0*}(\phi),\phi}$ is another viscosity solution for the same PDE, we conclude that $J^0_{\phi^{0*}(\phi),\phi} = V^0_{\phi} = \inf_{\balpha^0\leftrightarrow\phi^0} J^0_{\phi^0, \phi}$ and hence the optimality of $\phi^{0*}(\phi)$. Likewise we can show that $\bphi^*(\phi^0,\phi)$ is the best response of the representative minor player.
\end{proof}
In order to show that the Nash Equilibrium is actually given by a couple of Lipschitz feedback functions, we need an additional assumption on the regularity of value functions.

\begin{hypothesis}\label{assumpopt4}
There exists two constants $L_{\phi^0}, L_{\phi}$, such that for all $L_{\phi^0}$-Lipschitz feedback function $\phi^0$ and $L_{\phi}$-Lipschitz feedback function  $\phi$, $V^0_{\phi}$ is $(L_{\phi^0}/L_{\alpha^0} - 1)$-Lipschitz and $V_{\phi^0,\phi}$ is $(L_{\phi}/L_\alpha - 1)$-Lipschitz.
\end{hypothesis}

The above assumption holds, for example, when the horizon of the game is sufficiently small. We shall provide more details (see Remark \ref{remarkassumptionlipschitz} below) after we reveal important connections between finite player games and mean field games in Section \ref{se:chaos}. We now state and prove existence of Nash equilibrium.

\begin{theorem}
Under Hypothesis \ref{assumpopt1} - \ref{assumpopt4}, there exists a Nash equilibrium in the sense that there exists Lipschitz feedback functions $\hat\phi^0$ and $\hat\phi$ such that:
\[
[\hat\phi^0,\hat\phi] = [\bphi^{0*}(\hat\phi), \bphi^*(\hat\phi^0, \hat\phi)].
\]
\end{theorem}

\begin{proof}
We apply Schauder's fixed point theorem. To this end, we need to: (i) specify a Banach space $\mathbb{V}$ containing the admissible feedback functions $(\phi^0, \phi)$ as elements, and a relatively compact convex subset $\mathbb{K}$ of $\mathbb{V}$; (ii) show that the mapping 
$\mathcal{R}: [\phi^0,\phi] \rightarrow [\bphi^{0*}(\phi), \bphi^*(\phi^0, \phi)]$ is continuous and leaves $K$ invariant (i.e. $\mathcal{R}(K) \subset K$). 

\vspace{3mm}
\noindent (i) Define $\mathbb{C}^0$ as the collection of $A_0$ - valued functions $\phi^0$ on $[0,T]\times\{1,\dots,M^0\}\times\mathcal{P}$ such that $(t,x)\times \phi^0(t,i^0,x)$ is continuous for all $i^0$, $\mathbb{C}$ as the collection of $A$ - valued functions $\phi$ on $[0,T]\times\{1,\dots,M\}\times\{1,\dots,M^0\}\times\mathcal{P}$ such that $(t,x)\times \phi(t,i,i^0,x)$ is continuous for all $i,i^0$, and set $\mathbb{V} := \mathbb{C}^0 \times \mathbb{C}$. For all $(\phi^0,\phi)\in\mathbb{V}$, we define the norm:
\[
\|(\phi^0,\phi)\| := \max\left\{\sup_{i^0,t\in[0,T],x\in\mathcal{P}} |\phi^0(t,i^0,x)|, \sup_{i, i^0,t\in[0,T],x\in\mathcal{P}} |\phi(t,i,i^0,x)|\right\}.
\]
It is easy to check that $(\mathbb{V}, \|\cdot\|)$ is a Banach space. Next, we define $\mathbb{K}$ as the collection of elements in $\mathbb{V}$ such that the mappings $(t,x)\rightarrow \phi^0(t,i^0,x)$ are $L^0-$Lipschitz and $(t,x)\rightarrow \phi(t,i,i^0,x)$ are $L-$Lipschitz in $(t,x)$ for all $i^0=1,\dots,M^0$ and $i=1,\dots,M$, where $L^0, L$ are specified in Hypothesis \ref{assumpopt4}. Clearly $\mathbb{K}$ is convex. Now consider the family $(\phi^0(\cdot, i^0, \cdot))_{(\phi^0,\phi)\in \mathbb{K}}$ of functions defined on $[0,T]\times\mathcal{P}$. Thanks to the Lipschitz property, we see immediately that the family is equicontinuous and pointwise bounded. Therefore by Arzel\`a-Ascoli theorem, the family is compact with respect to the uniform norm. Repeating this argument for all $i,i^0$ we see that $\mathbb{K}$ is compact under the norm $\|\cdot\|$. Moreover, thanks to Hypothesis \ref{assumpopt2} - \ref{assumpopt4}, we obtain easily that $\mathbb{K}$ is stable by $\mathcal{R}$.

\vspace{3mm}
\noindent(ii) It remains to show that $\mathcal{R}$ is a continuous mapping.  We use the following lemma:

\begin{lemma}\label{lemexistencenash}
Let $(\phi^0_n, \phi_n)$ be a sequence in $\mathbb{K}$ converging to $(\phi^0, \phi)$ in $\|\cdot\|$, and denote by $V^{0}_n$ and $V_n$ the value functions of the major and representative minor players associated with $(\phi^0_n, \phi_n)$. Then $V^{0}_n$ and $V_n$ converge uniformly to $V^0$ and $V$ respectively where $V^0$ and $V$ are the value functions of the major player and the representative minor player associated with $(\phi^0, \phi)$.
\end{lemma}

The proof of the lemma uses standard arguments from the theory of viscosity solutions. We give it in the Appendix. Now the continuity of the mapping $\mathcal{R}$ follows readily from Lemma \ref{lemexistencenash}, Proposition \ref{optimalresponseprop} and Hypothesis \ref{assumpopt2} \& \ref{assumpopt3}. This completes the proof.
\end{proof}

\section{The Master Equation and the Verification Argument for Nash Equilibria}
\label{se:master}
If a Nash equilibrium exists and is given by feedback functions $\hat\phi^0$ for the major player  and $\hat\phi$ for the minor players,
these functions should also be equal to the respective minimizers of the Hamiltonians in the HJB equations of the optimization problems. This informal remark leads to a system of coupled PDEs with terminal conditions specified at $t=T$, which we expect to hold if the equilibrium exists. Now the natural question to ask is: if this system of PDEs has a solution, does this solution provide a Nash equilibrium? The following result provides a verification argument:

\begin{theorem}\label{verificationtheo}
(Verification Argument) Assume that there exists two function $\hat\phi^0: [0,T] \times \{1,\dots, M^0\} \times \mathcal{P}\ni(t, i^0, x) \rightarrow \hat\phi^0(t, i^0, x) \in\mathbb{R}$ and $\hat\phi: [0,T] \times \{1,\dots, M\} \times \{1,\dots, M^0\} \times \mathcal{P}\ni(t, i, i^0, x) \rightarrow \hat\phi(t, i, i^0, x)\in\mathbb{R}$ such that the system of PDEs in $(v^0, v)$:
\begin{equation}
\label{hjbmaster}
\begin{aligned}
&0=[\mathcal{G}^0_{\hat\phi^0,\hat\phi} v^0](t, i^0, x) + f^0(t, i^0, \hat\phi^0(t,i^0, x), x)\\
&\hskip 75pt
v^0(T,i^0, x) = g^0(i^0, x),\qquad (i^0,x)\in\{1,\dots, M^0\}\times \mathcal{P}\\
&0 = [\mathcal{G}_{\hat\phi^0,\hat\phi,\hat\phi} v](t, i,i^0, x) + f(t,i, \hat\phi(t, i,i^0, x), i^0, \hat\phi^0(t,i^0, x),  x)\\
&\hskip 75pt
v(T,i,i^0, x) = g(i,i^0, x),\qquad (i,i_0,x)\in  \{1,\dots, M\}\times\{1,\dots, M^0\}\times\mathcal{P}
\end{aligned}
\end{equation}
admits a classical solution $(\hat V^{0}, \hat V)$ (i.e. the solution are $\mathcal{C}^1$ in $t$ and $x$). Assume in addition that:
\begin{equation}
\label{minimizers}
\begin{aligned} 
&\hat\phi^0(t,i^0,x) = \hat\alpha^0(t, i^0, x, \hat V^0(t,\cdot,x))\\
&\hat\phi(t,i,i^0,x) = \hat\alpha(t,i, i^0, x, \hat V(t,\cdot,\cdot,x))
\end{aligned}
\end{equation}
Then $\hat\phi^0$ and $\hat\phi$ form a Nash equilibrium and $\hat V^{0}(0,X_0^0, \mu_0)$ and $\hat V(0, X_0, X_0^0, \mu_0)$ are the equilibrium expected costs of the major and minor players.
\end{theorem}

\begin{proof}
We show that $\hat\phi^0 = \bphi^{0*}(\hat \phi)$ and $\hat\phi = \bphi^*(\hat \phi^0, \hat \phi)$. Notice first that $\hat\phi^0$ and $\hat\phi$ are Lipschitz strategies due to the regularity of $\hat V^0, \hat V$ and Hypothesis \ref{assumpopt2}-\ref{assumpopt3}.

Consider the major player optimization problem where we let $\phi = \hat \phi$ and denote by $V^0_{\hat\phi}$ the corresponding value function. Then since $\hat V^0$ is a classical solution to (\ref{hjbmaster}) and because of (\ref{minimizers}), we deduce that $\hat V^0$ is a viscosity solution to the HJB equation (\ref{hjbmajor}) associated with the value function $V^0_{\hat\phi}$. By uniqueness of the viscosity solution, we conclude that $\hat V^0 = V^0_{\hat\phi}$.

On the other hand, if we denote by $J^0_{\hat\phi^0,\hat\phi}$ the expected cost function of the major player when it uses the feedback function $\hat\phi^0$ and all the minor players use strategy $\hat\phi$, then the fact that $\hat V^0$ is a classical solution to (\ref{hjbmaster}) implies that $\hat V^0$ is also a viscosity solution. Then by Proposition \ref{pdecharacterizepayoff} we have $J^0_{\hat\phi^0,\hat\phi} = \hat V^0$ and therefore $J^0_{\hat\phi^0,\hat\phi} = V^0_{\hat\phi} = \inf_{\balpha^0\leftrightarrow\phi^0} J^0_{\phi^0,\hat\phi}$. This means that $\hat \phi^0$ is the best response of the major player to the minor players using feedback function $\hat\phi$.

\vskip 2pt
For the optimization problem of the representative minor player, we use the same argument based on the uniqueness of solution of PDE to obtain  $J_{\hat\phi^0,\hat\phi,\hat\phi}= \hat V = V_{\hat\phi^0,\hat\phi} = \inf_{\bar\balpha\leftrightarrow\bar\phi} J_{\hat\phi^0,\hat\phi,\bar\phi}.$
This implies that $\hat\phi$ is the representative player's best response to the major player using feedback function $\hat\phi^0$ and the rest of the minor players using $\hat\phi$. We conclude that $\hat\phi^0$ and $\hat\phi$ form the desired fixed point for the best response map.
\end{proof}

It is important to keep in mind that the above verification argument of the Master equation does not speak to the problem of existence of Nash equilibria. However, it provides a convenient way to compute numerically the equilibrium via the solution of a coupled system of first-order PDEs.

\section{Propagation of Chaos and Approximate Nash Equilibria}
\label{se:chaos}

In this section we show that in the $(N+1)$-player game (see description in Section \ref{se:finite}), when the major player and each minor player apply the respective equilibrium strategy in the mean field game, the system is in an approximate Nash equilibrium. To uncover this link, we first revisit the $(N+1)$-player game. We show that for a certain strategy profile, the expected cost of individual player in the finite player game converges to that of the mean field game. Our argument is largely similar to the one used in proving the convergence of numerical scheme for viscosity solutions. One crucial intermediate result we use here is the gradient estimate for the value functions of the $(N+1)$-player game. Similar results were proved in \cite{GomesMohrSouza_continuous} for discrete state mean field game without major player. As a biproduct of the proof, we can also conclude that the value function of the mean field game is Lipschitz in the measure argument. In the rest of the section, we assume that Hypothesis \ref{assumpopt1} is in force.

\subsection{Back to the $(N+1)$-Player Game}
In this section, we focus on the game with a major player and $N$ minor players. We show that both the expected costs of individual players and the value functions of the players' optimization problems can be characterized by coupled systems of ODEs, and their gradients are bounded by some constant independent of $N$. Such a gradient estimate will be crucial in establishing results on propagation of chaos, as well as the regularity of the value functions for the limiting mean field game.

We start from the major player's optimization problem. Consider a strategy profile where the major player chooses a Lipschitz feedback function $\phi^0$ and all the $N$ minor players choose the same Lipschitz feedback function $\phi$. Recall that the process comprising the major player's state and the empirical distirbution of the states of the minor players, say $(X_t^{0,N}, \mu_t^N)$, is a finite-state Markov process in the space $\{1,\dots, M^0\} \times \mathcal{P}^N$, where $\mathcal{P}^N := \{\frac{1}{N}(k_1, \dots, k_{M-1}) | \sum_{i} k_i \le N, k_i \in \mathbb{N}\}$. Its infinitesimal generator $\mathcal{G}_{\phi^0,\phi}^{0, N}$ was given by (\ref{generatorNmajor}).  The expected cost to the  major player is given by:
\[
J^{0,N}_{\phi^0,\phi}(t,i^0,x) := \mathbb{E}\left[\int_{t}^{T} f^0_{\phi^0}(s, X_s^{0,N}, \mu_s^N) ds + g^0(X_T^{0,N}, \mu_T^N) | \; X_t^{0,N} = i^0, \mu_t^N = x\right]
\]
and the value function of the major player's optimization problem by:
\[
V^{0,N}_{\phi}(t,i^0,x) := \inf_{\balpha^0\leftrightarrow\phi^0 \in\mathbb{A}^0}J^{0,N}_{\phi^0,\phi}(t,i^0,x).
\]
Despite the notation, $J^{0,N}_{\phi^0,\phi}$ can be viewed as a function defined on $[0,T]$ with values given by vectors indexed by $(i^0, x)$. The following result shows that $J^{0,N}_{\phi^0,\phi}$ is characterized by a coupled system of ODEs.
\begin{proposition}\label{odemajorpayoffprop}
Let $\phi^0\in\mathbb{L}^0$ and $\phi\in\mathbb{L}$, then $J^{0,N}_{\phi^0,\phi}$ is the unique classical solution of the system of ODEs:
\begin{equation}\label{odemajorpayoff}
\begin{aligned}
0 =&\; \dot\theta(t,i^0,x) + f^0_{\phi^0}(t, i^0, x) + \sum_{j^0, j^0 \neq i^0} (\theta(t, j^0, x) - \theta(t, i^0, x))   q^0_{\phi^0}(t, i^0, j^0, x)\\
&\;+ \sum_{(i,j), j\neq i}(\theta(t, i^0, x + \frac{1}{N}e_{ij} ) -\theta(t, i^0, x )) N x_i  q_{\phi^0, \phi}(t, i, j, i^0,  x)\\
0 =&\; \theta(t,i^0,x) - g^0(i^0, x)
\end{aligned}
\end{equation}
\end{proposition}
\begin{proof}
The existence and uniqueness of the solution to (\ref{odemajorpayoff}) is an easy consequence of the Lipschitz property of the functions $f^0, q, q^0, \phi^0,\phi$ and Cauchy-Lipschitz Theorem. The fact that $J^{0,N}_{\phi^0,\phi}$ is a solution to (\ref{odemajorpayoff}) follows from Dynkin formula. \end{proof}

We state without proof the similar result for $V^{0,N}_{\beta}$.

\begin{proposition}\label{odemajorvalueprop}
If Hypotheses \ref{assumpopt1} - \ref{assumpopt3} hold and $\phi$ is a  Lipschitz strategy, then $V^{0,N}_{\phi}$ is the unique classical solution of the system
of ODEs:
\begin{equation}\label{odemajorvalue}
\begin{aligned}
0 =&\; \dot\theta(t,i^0,x) + \inf_{\alpha^0 \in A^0} \{ f^0(t, i^0,\alpha^0, x) + \sum_{j^0, j^0 \neq i^0} (\theta(t, j^0, x) - \theta(t, i^0, x))   q^0(t, i^0, j^0, \alpha^0, x)\}\\
&+ \sum_{(i,j), j\neq i}(\theta(t, i^0, x + \frac{1}{N}e_{ij} ) -\theta(t, i^0, x )) N x_i  q(t, i, j, \phi(t, i, i^0, x), i^0, x)\\
0 =&\; \theta(t,i^0,x) - g^0(i^0, x)
\end{aligned}
\end{equation}
\end{proposition}

The following estimates for $J^{0,N}_{\phi^0,\phi}$ and $V^{0,N}_{\phi}$ will play a crucial role in proving convergence to the solution of the mean field game. Their proofs are postponed to the appendix.

\begin{proposition}\label{majorNproperty}
For all Lipschitz strategies $\phi^0,\phi$, there exists a constant $L$ only depending on $T$ and the Lipschitz constants and bounds of $\phi^0,\phi, q^0,q,f^0, g^0$ such that for all $N>0$, $(t,i^0, x)\in [0,T]\times \{1,\dots, M^0\}\times\mathcal{P}^N$ and $j,k\in\{1,\cdots,M\}, j\neq k$, we have:
\[
|J^{0,N}_{\phi^0,\phi}(t,i^0,x)|\le \|g^0\|_{\infty} + T \|f^0\|_{\infty},
\quad\text{and}\quad
J^{0,N}_{\phi^0,\phi}(t,i^0,x+\frac{1}{N}e_{jk}) - J^{0,N}_{\phi^0,\phi}(t,i^0,x)|\le \frac{L}{N}.
\]
\end{proposition}
\begin{proposition}\label{majorvalueodeproperty}
For each Lipschitz feedback function $\phi$ with Lipschitz constant $L_\phi$, there exists constants $C_0, C_1, C_2, C_3, C_4 >0$ only depending on $M^0$, $M$, the Lipschitz constants and bounds of $q^0,q,f^0, g^0$, such that for all  $N>0$, $(t,i^0, x)\in [0,T]\times \{1,\dots, M^0\}\times\mathcal{P}^N$ and $j,k\in\{1,\cdots,M\}, j\neq k$, we have:
\[
|V^{0,N}_{\phi}(t,i^0,x+\frac{1}{N}e_{jk}) - V^{0,N}_{\phi}(t,i^0,x)|\le \frac{C_0 + C_1 T + C_2 T^2}{N}\exp[(C_3 + C_4 L_\phi)T].
\]
\end{proposition}

We now turn to the problem of the representative minor player. We consider a strategy profile where the major player uses a feedback function $\phi^0$, the first $(N-1)$ minor players use a feedback function $\phi$ and the remaining (de facto the representative) minor player uses the feedback function $\o\phi$. We recall that $(X_t^N, X_t^{0,N}, \mu_t^N)$ is a Markov process with infinitesimal generator $\mathcal{G}^{N}_{\phi^0,\phi,\o\phi}$ defined as in (\ref{generatorNminor}). We are interested in the representative minor player's expected cost:
\begin{align*}
J^{N}_{\phi^0,\phi,\o\phi}(t,i,i^0,x) :=&\; \mathbb{E}\bigg[\int_{t}^{T} f_{\phi^0, \o\phi}(s,X_s^N, X_s^{0,N},\mu_s^N) ds + g(X_T^N, X_T^{0,N}, \mu_T^N) | \; X_t^N = i, X_t^{0,N} = i^0, \mu_t^N = x\bigg]
\end{align*}
as well as the value function of the representative minor player's optimization problem:
\[
V^{N}_{\phi^0,\phi}(t,i,i^0,x) := \sup_{\alpha\leftrightarrow\o\phi\in\mathbb{A}}J^{N}_{\phi^0,\phi,\o\phi}(t,i,i^0,x).
\]
In full analogy with propositions \ref{majorNproperty} and \ref{majorvalueodeproperty}, we state the following results without proof.

\begin{proposition}\label{minorNproperty}
For all Lipschitz feedback functions $\phi^0$ and $\phi$, there exists a constant $L$ only depending on $T$ and the Lipschitz constants and bounds of $\phi^0,\phi,\o\phi, q^0,q,f, g$ such that for all $N>0$, $(t,i^0, x)\in [0,T]\times \{1,\dots, M^0\}\times\mathcal{P}^N$ and $j,k\in\{1,\cdots,M\}, j\neq k$, we have:
\[
|J^{N}_{\phi^0,\phi,\o\phi}(t,i,i^0,x)|\le \|g\|_{\infty} + T \|f\|_{\infty}, 
\quad\text{and}\quad
J^{N}_{\phi^0,\phi,\o\phi}(t,i,i^0,x+\frac{1}{N}e_{jk}) - J^{N}_{\phi^0,\phi,\o\phi}(t,i,i^0,x)|\le \frac{L}{N}.
\]
\end{proposition}

\begin{proposition}\label{minorvalueodeproperty}
There exist constants $D_0, D_1, D_2, D_3, D_4, D_5 >0$  depending only on $M^0$, $M$, the Lipschitz constants and bounds of $q^0,q,f, g$ such that for all Lipschitz feedback functions $\phi^0$ and $\phi$ with Lipschitz constants $L_{\phi^0}$ and $L_\phi$ respectively, and for all $N>0$, $(t,i^0, x)\in [0,T]\times \{1,\dots, M^0\}\times\mathcal{P}^N$ and $j,k\in\{1,\cdots,M\}, j\neq k$, we have:
\[
|V^{0,N}_{\phi^0,\phi}(t,i,i^0,x+\frac{1}{N}e_{jk}) - V^{0,N}_{\phi^0,\phi}(t,i,i^0,x)|\le 
 \frac{D_0 + D_1 T + D_2 T^2 + D_3 L_{\phi^0} T}{N}\exp[(D_4 + D_5 L_\phi))T].
\]
\end{proposition}

\subsection{Propagation of Chaos}

We now prove two important limiting results.
They are related to the propagation of chaos in the sense that they identify the limiting behavior of an individual when interacting with the mean field.
First, we prove uniform convergence of the value functions of the individual players' optimization problems. Combined with the gradient estimates
proven in the previous subsection, this establishes the Lipschitz property of the value functions in the mean field limit. Second, we prove that the  expected costs of the individual players in the $(N+1)$ - player game converge to their mean field limits at the rate $N^{-1/2}$. This will help us show that the Nash equilibrium of the mean field game provides approximative Nash equilibria for the finite player games. 

\begin{theorem}\label{conv1}
For all Lipschitz strategy $\phi^0,\phi$, we have:
\begin{align}
\sup_{t,i^0,x} |V^0_{\phi}(t,i^0,x) - V^{0,N}_{\phi}(t,i^0,x)| \rightarrow 0,& \;\;\;N\rightarrow +\infty \label{convmajorvalue}\\
\sup_{t,i,i^0,x} |V_{\phi^0,\phi}(t,i,i^0,x) - V^{N}_{\phi^0,\phi}(t,i,i^0,x)| \rightarrow 0,& \;\;\;N\rightarrow +\infty \label{convminorvalue}
\end{align}
\end{theorem}
\begin{proof}
We only provide a proof for (\ref{convmajorvalue}) as (\ref{convminorvalue}) can be shown in the exact the same way.

\noindent (i) Fix a Lipschitz strategy $\phi$ for the minor players. To simplify the notation, we set $v_N := V_{\phi}^{0, N}$. Notice that $(t,i^0,x)\rightarrow v_N(t,i^0,x)$ is only defined on $[0,T]\times\{1,\dots,M^0\}\times \mathcal{P}^N$, so our first step is to extend the domain of $v^N(t,i^0,x)$ to $[0,T]\times\{1,\dots,M^0\}\times\mathcal{P}$. This can be done by considering the linear interpolation of $v^N$. More specifically, for any $x\in\mathcal{P}$, we denote $x_k, k=1,\dots, 2^{M-1}$ the $2^{M-1}$ closest neighbors of $x$ in the set $\mathcal{P}^N$. There exists $\alpha_k, k=1,\dots,2^{M-1}$ positive constants such that $x = \sum_{k=1}^{2^{M-1}} \alpha_k x_k$. We then define the extension, still denoted as $v_N$, to be $v^N(t,i^0,x):=\sum_{k=1}^{2^{N-1}} \alpha_k v_N(t,i^0,x_k)$. It is straightforward to verify that $v_N$ is continuous in $(t,x)$, Lipschitz in $x$ uniformly in $(t,i^0)$, and $C^1$ in $t$.  Using the boundedness and Lipschitz property of $v_N, f^0, q^0, q,\phi$, we obtain a straight forward estimation:
\begin{equation}\label{convest1}
\begin{aligned}\allowdisplaybreaks
\frac{L}{N} \ge&\; \bigl |\dot v_N(t,i^0,x) + \inf_{\alpha^0\in A^0}\{f^0(t, i^0, \alpha^0, x) + \sum_{j^0, j^0 \neq i^0} (v_N(t, j^0, x) - v_N(t, i^0, x))   q^0(t, i^0, j^0, \alpha^0, x)\}\\
&+ \sum_{(i,j), j\neq i}(v_N(t, i^0, x + \frac{1}{N}e_{ij} ) -v_N(t, i^0, x )) N x_i  q(t, i, j, \phi(t, i, i^0, x), i^0, x)\bigr |
\end{aligned}
\end{equation}
\begin{equation}\label{convest2}
\frac{L}{N} \ge \; |v_N(t,i^0,x) - g^0(i^0, x)|
\end{equation}
where the constant $L$ only depends on the bounds and Lipschitz constants of $f^0, q^0, q$ and $\phi$.

\vspace{3mm}
\noindent (ii) Now let us denote $\bar v(t,i^0,x):=\lim\sup^*v_N(t,i^0,x)$ and $\munderbar v_N(t,i^0,x):=\lim\inf_*v_N(t,i^0,x)$, see Section 9.3 for definitions of the operators $\lim\sup^*$ and $\lim\inf_*$. We show that $\bar v$ and $\munderbar v$ are  viscosity subsolution and viscosity supersolution of the HJB equation (\ref{hjbmajor}) of the major player
respectively. Recall that we assume now that $q$ does not depend on $\alpha$. Then since $V_{\phi}^{0}$ is also a viscosity solution to (\ref{hjbmajor}), the comparison principle allows us to conclude that $\bar v(t,i^0,x) = \munderbar v(t,i^0,x) = V_{\phi}^{0}$ and the uniform convergence follows by standard arguments.

\vspace{3mm}
\noindent (iii) It remains to show that $\bar v$ is a viscosity subsolution to the PDE (\ref{hjbmajor}). The proof of $\munderbar v$ being a viscosity supersolution can be done in exactly the  same way. Let $\theta$ be a smooth function and $(\bar t,i^0,\bar x)\in [0,T]\times \{1,\dots, M^0\} \times \mathcal{P}$ be such that $(t,x)\rightarrow\bar v(t,i^0,x) - \theta(t,x)$ has maximum at $(\bar t,\bar x)$ and $\bar v(\bar t,i^0,\bar x) =\theta(\bar t,\bar x)$. Then by Lemma 6.1. in \cite{lions}, there exists sequences $N_n\rightarrow +\infty$, $t_n\rightarrow \bar t$, $x_n\rightarrow \bar x$ such that for each $n$, the mapping $(t,x) \rightarrow v_{N_n}(t,i^0,x) - \theta(t,x)$ attains a maximum at $t_n, x_n$ and $\delta_n :=  v_{N_n}(t_n,i^0,x_n) - \theta(t_n,x_n)\rightarrow 0$. Instead of extracting a subsequence, we may assume that $v_{N_n}(t_n,\cdot,x_n) \rightarrow (r_1, \dots, r_{M^0})$, where $r_{j^0} \le \bar v(\bar t,j^0,\bar x)$ and $r_{i^0} = \bar v(\bar t,i^0,\bar x)$.

Assume that $\bar t=T$, then $\bar v (\bar t, i^0, \bar x) \le g^0(i^0, \bar x)$ follows easily from (\ref{convest2}). Now assume that $\bar t<T$. Instead of extracting a subsequence, we may assume that $t_n<T$ for all $n$. Then by maximality we have $\partial_t \theta(t_n,x_n) = \partial_t v_{N_n}(t_n,x_n)$. Again by maximality, we have for all $i,j = 1,\dots, M, i\neq j$:
\[
v_{N_n}(t_n, i^0, x_n + \frac{1}{N_n}e_{i,j}) - v_{N_n}(t_n, i^0, x_n) \le \theta(t_n, i^0, x_n + \frac{1}{N_n}e_{i,j}) - \theta(t_n, i^0, x_n).
\]
Injecting the above inequalities into the estimation (\ref{convest1}) and using the postivity of $q$, we obtain:
\begin{align*}
-\frac{L}{N_n} \le&\; \partial_t\theta(t_n,i^0,x_n) + \inf_{\alpha^0\in A^0}\bigg\{f^0(t_n, i^0, \alpha^0, x_n) + \sum_{j^0, j^0 \neq i^0} (v_{N_n}(t_n, j^0, x_n) - v_{N_n}(t_n, i^0, x_n))   q^0(t_n, i^0, j^0, \alpha^0, x_n)\bigg\}\\
&+ \sum_{(i,j), j\neq i}(\theta(t_n, i^0, x_n + \frac{1}{N_n}e_{ij} ) -\theta(t_n, i^0, x_n )) N_n (x_n)_i  q(t_n, i, j, \phi(t_n, i, i^0, x_n), i^0, x_n).
\end{align*}
Taking limit in $n$ we obtain:
\begin{align*}
0 \le&\; \partial_t\theta(\bar t,i^0,\bar x) + \inf_{\alpha^0\in A^0}\{f^0(\bar t, i^0, \alpha^0, \bar x) + \sum_{j^0, j^0 \neq i^0} (r_{j^0} - r_{i^0})   q^0(\bar t, i^0, j^0, \alpha^0, \bar x)\}\\
&+ \sum_{(i,j), j\neq i}(\mathbbm{1}_{j\neq M} \partial_{x_j}\theta(\bar t, i^0, \bar x)- \mathbbm{1}_{i\neq M} \partial_{x_i}\theta(\bar t, i^0, \bar x)) \bar x_i  q(\bar t, i, j, \phi(\bar t, i, i^0, \bar x), i^0, \bar x)
\end{align*}
Now since $q^0$ is positive and $r_{j^0} \le \bar v(\bar t,j^0,\bar x)$ for all $j^0 \neq i^0$ and $r_{i^0} = \bar v(\bar t,i^0,\bar x)$ we have:
\begin{align*}
&\;\;\inf_{\alpha^0\in A^0}\{f^0(\bar t, i^0, \alpha^0, \bar x) + \sum_{j^0, j^0 \neq i^0} (r_{j^0} - r_{i^0})   q^0(\bar t, i^0, j^0, \alpha^0, \bar x)\} \\
\le&\;\;\inf_{\alpha^0\in A^0}\{f^0(\bar t, i^0, \alpha^0, \bar x) + \sum_{j^0, j^0 \neq i^0} (\o v(\bar t,j^0,\bar x) - \o v(\bar t,i^0,\bar x))   q^0(\bar t, i^0, j^0, \alpha^0, \bar x)\}.
\end{align*}
The desired inequality for viscosity subsolution follows immediately. This completes the proof.
\end{proof}
As an immediate consequence of the uniform convergence and the gradient estimates for the value functions $V_{\phi}^{0,N}$ and $V_{\phi^0,\phi}^{N}$, we have:

\begin{corollary}\label{valuefunctionlipschitzconstant}
Under hypotheses \ref{assumpopt1}-\ref{assumpopt3}, for all Lipschitz feedback functions $\phi^0$ and $\phi$ with Lipschitz constants $L_{\phi^0}$ and $L_\phi$ respectively, the value functions $V_{\phi}^0$ and $V_{\phi^0,\phi}^0$ are Lipschitz in $(t,x)$. More specifically, there exist strictly positive constants $B$, $C_i, i=0,\dots,4$, $D_i, i= 0,\dots, 5$ that only depend on the bounds and Lipschitz constants of $f,f^0,g^0,g,q,q^0$ such that
\begin{align}
|V_{\phi}^0(t,i^0,x) - V_{\phi}^0(s,i^0,y)| \le& B |t-s| + (C_0 + C_1 T + C_2 T^2)\exp((C_3 + C_4 L_\phi)T)\|x - y\|\\
|V_{\phi^0,\phi}(t,i,i^0,x) - V_{\phi^0,\phi}(s,i,i^0,y)| \le& B |t-s| + (D_0 + D_1 T + D_2 T^2 + D_3 L_{\phi^0} T)\exp((D_4 + D_5 L_\phi)T)\|x - y\|.
\end{align}
\end{corollary}
\begin{proof}
The Lipschitz property in $x$ is an immediate consequence of Theorem \ref{conv1}, Proposition \ref{majorvalueodeproperty} and Proposition \ref{minorvalueodeproperty}. To prove the Lipschitz property on $t$, we remark that for each $N$, $V^{0,N}_{\phi}$ and $V^{N}_{\phi^0,\phi}$ are Lipschitz in $t$, uniformly in $N$. Indeed $V^{0,N}_{\phi}$ is a classical solution of the system \eqref{odemajorvalue} of ODEs. Then it is clear that $\frac{d}{dt}V^{0,N}_{\phi}$ is bounded by the bounds of $f^0,g^0,q^0,q$. We deduce that $V^{0,N}_{\phi}$ is Lipschitz in $t$ with a Lipschitz constant that only depends on $M, M^0$ and the bounds of $f^0,g^0,q^0,q$. By convergence of $V^{0,N}_{\phi}$, we conclude that $V^{0}_{\phi}$ is also Lipschitz in $t$ and shares the  Lipschitz constant of $V^{0,N}_{\phi}$. The same argument applies to $V_{\phi^0,\phi}$.
\end{proof} 
\begin{remark}\label{remarkassumptionlipschitz}
From Corollary \ref{valuefunctionlipschitzconstant}, we see that Hypothesis \ref{assumpopt4} holds when $T$ is sufficiently small. Indeed we can first choose $L_{\phi^0} > L_a (1 + \max\{B, C_0\})$ and $L_\phi > L_b (1 + \max\{B, D_0\})$ and then choose $T$ sufficiently small, so that the Lipschitz constant of $V_{\phi}^0$ is smaller than $(L_{\phi^0}/L_a-1)$ and the Lipschitz constant of $V_{\phi^0, \phi}$ is smaller than $(L_\phi/L_b-1)$.
\end{remark}

We now state our second result on the propagation of chaos: the expected cost of an individual player in the $(N+1)$-player game converges to the expected cost in the mean field game at a rate of $N^{-1/2}$.

\begin{theorem}\label{propachaos}
There exists a constant $L$ depending only on $T$ and the Lipschitz constants of $\phi^0$, $\phi$, $\o\phi$, $f$, $f^0$, $g$, $g^0$, $q$ and $q^0$ such that for all $N>0$, $t\le T$, $x\in\mathcal{P}^N$, $i=1,\dots, M$ and $i^0=1,\dots,M^0$, we have
\begin{align*}
|J_{\phi^0,\phi}^{0,N}(t,i^0,x) - J_{\phi^0,\phi}^0(t,i^0,x)| \le& L/\sqrt{N}\\
|J_{\phi^0,\phi,\o\phi}^{N}(t,i,i^0,x) - J_{\phi^0,\phi,\o\phi}(t,i,i^0,x)| \le& L/\sqrt{N}.
\end{align*}
\end{theorem}

Our proof is based on standard techniques from the convergence rate analysis of numerical schemes for viscosity solutions of PDEs, c.f. \cite{BrianiCamilliZidani} and \cite{BarlesSouganidis} for example. The key step of the proof is the construction of a smooth subsolution and a smooth supersolution of the PDEs (\ref{majorJpde}) and (\ref{minorJpde}) that characterizing $J_{\phi^0,\phi}^0$ and $J_{\phi^0,\phi,\o\phi}^0$ respectively. See Proposition \ref{pdecharacterizepayoff}. We construct these solutions by mollifying an extended version of $J_{\phi^0,\phi}^{0,N}$. Then we derive the bound by using the comparison principle. In the following we detail the proof for the convergence rate of the major player's expected cost. The case of the generic minor player can be dealt with in exactly the same way.

Since $J_{\phi^0,\phi}^{0,N}(t,i,x)$ is only defined for $x\in\mathcal{P}^N$, in order to mollify $J_{\phi^0,\phi}^{0,N}$, we need to first construct an extension of $J_{\phi^0,\phi}^{0,N}$ defined for all $x\in O$ for an open set $O$ containing $\mathcal{P}$. To this end, we consider the following system of ODE:
\begin{align*}
0 =&\; \dot\theta(t,i^0,x) + \tilde f^0(t, i^0, \tilde\phi^0(t, i^0, x), x) + \sum_{j^0, j^0 \neq i^0} (\theta(t, j^0, x) - \theta(t, i^0, x))   \tilde q^0(t, i^0, j^0, \tilde\phi^0(t, i^0, x), x)\\
&+ \sum_{(i,j), j\neq i}(\theta(t, i^0, x + \frac{1}{N}e_{ij} ) -\theta(t, i^0, x )) N \max\{x_i, 0\}  \tilde q(t, i, j, \tilde\phi(t, i, i^0, x), i^0, \tilde\phi^0(t,i^0,x), x)\\
0 =&\; \theta(T,i^0,x) - \tilde g^0(i^0, x)
\end{align*}
Here $\tilde\phi^0$, $\tilde\phi$, $\tilde f^0$, $\tilde g^0$ and $\tilde q^0$ are respectively extensions of $\phi^0$, $\phi$, $f^0$, $g^0$ and $q^0$ from $x\in\mathcal{P}$ to $x\in\mathbb{R}^{N-1}$, which are Lipschitz in $x$. The following is proved using the same arguments as for Proposition \ref{odemajorpayoffprop}.

\begin{lemma}\label{propertyextension}
The system \eqref{odemajorpayoff} of ODEs admits a unique solution $v^N$ defined in $[0,T]\times \{1,\dots,M^0\}\times \mathbb{R}^{M-1}$. Moreover we have:

\noindent (i) $v^N(t,i^0,x)$ is Lipschitz in $x$ uniformly in $t$ and $i^0$ and the Lipschitz constant only depends on $T$ and the Lipschitz constants of $\phi^0$, $\phi$,$f^0$, $g^0$, $q^0$.

\noindent (ii) $v^N(t,i^0,x) = J_{\phi^0,\phi}^{0,N}(t,i^0,x)$ for all $x \in \mathcal{P}^N$.
\end{lemma}

To construct smooth super and sub solutions, we use a family of mollifiers $\rho_\epsilon$ defined by $\rho_\epsilon(x) := \rho(x/\epsilon)/\epsilon^{N-1}$, where $\rho$ is a smooth and positive function with compact support in the unit ball of $\mathbb{R}^{M-1}$ and satisfying $\int_{R^{M-1}} \rho(x) dx = 1$. For $\epsilon>0$, we define $v^{N}_\epsilon$ as the mollification of $v^N$ on $[0,T]\times\{1,\dots, M^0\}\times \mathcal{P}$:
\[
v^{N}_\epsilon(t,i^0,x) := \int_{y\in\mathbb{R}^{M-1}} v^N(t,i^0,x-y)\rho_\epsilon(y) dy.
\]
Using the Lipschitz property of $\phi^0$, $\phi$, $f^0$, $g^0$, $q^0$ and straightforward estimates on the mollifier $\rho_\epsilon$, we obtain the following properties on $v^{N}_\epsilon$.
\begin{lemma}\label{molifyest}
$v^{N}_\epsilon$ is $C^{\infty}$ in $x$ and $C^1$ in $t$. Moreover, there exists a constant $C$ that depends only on $T$ and the Lipschitz constants of $\phi^0$, $\phi$, $f^0$, $g^0$ and $q^0$ such that for all $i^0=1,\dots, M^0$, $i=1,\dots, M$, $t\le T$ and $x,y\in\mathcal{P}$, the following estimations hold:
\begin{align}
L\epsilon \ge&\;| [\mathcal{G}_{\phi^0,\phi}^{0, N}v^{N}_\epsilon](t,i^0,x)+  f^0(t, i^0,\phi^0(t,i^0,x),x)|\label{molifyest1}\\
L\epsilon \ge&\; |v^{N}_\epsilon(T,i^0,x) - g^0(i^0, x)|\label{molifyest2}\\
\frac{L}{\epsilon}\|x-y\| \ge&\; |\partial_{x_i}v^{N}_\epsilon(t,i^0,x) - \partial_{x_i}v^{N}_\epsilon(t,i^0, y)|. \label{molifyest3}
\end{align}
\end{lemma}
We are now ready to prove Theorem \ref{propachaos}. We construct viscosity super and sub solutions by adjusting $v^{N}_\epsilon$ with a linear function on $t$. Then the comparison principle allows us to conclude.

\begin{proof}(of Theorem \ref{propachaos})
We denote by $L$ a generic constant that only depends on $T$ and the Lipschitz constants of $\phi^0$, $\phi$,$f^0$, $g^0$ and $q^0$. Using (\ref{molifyest1}), (\ref{molifyest2}) and (\ref{molifyest3}) in Lemma \ref{molifyest}, we obtain:
\begin{align*}
L(\epsilon + \frac{1}{N\epsilon})\ge\;&|[\mathcal{G}_{\phi^0,\phi}^{0}v^{N}_\epsilon](t,i^0,x)+  f^0(t, i^0,\phi^0(t,i^0,x),x)|\\
L\epsilon\ge\;&|v^{N}_\epsilon(T,i^0, x) - g^0(i^0, x)|.
\end{align*}
Next we define:
\[
v_{\pm}^{N}(t,i^0,x) :=v^{N}_\epsilon(t,i^0,x) \pm [ L(\epsilon + \frac{1}{N\epsilon})(T-t)+L\epsilon].
\]
Since $v_{-}^{N}$ and $v_{+}^{N}$ are smooth, the above estimation immediately implies that $v_{-}^{N}$ and $v_{+}^{N}$ are viscosity sub and super solutions of the PDE (\ref{majorJpde}) respectively. By Proposition \ref{pdecharacterizepayoff}, $J^0_{\phi^0,\phi}$ is a continuous viscosity solution to the PDE (\ref{majorJpde}). Then by the comparison principle we have $v_{-}^{N}\le J^0_{\phi^0,\phi} \le v_{+}^{N}$, which implies:
\[
|v^{N}_\epsilon(t,i^0,x) - J^0_{\phi^0,\phi}(t,i^0,x)| \le L(\epsilon + \frac{1}{N\epsilon})(T-t)+L\epsilon
\]
Now using the property of the mollifier and Lemma \ref{propertyextension}, we have for all $t\le T$, $i^0=1,\dots,M^0$ and $x\in\mathcal{P}^N$:
\[
|v^{N}_\epsilon(t,i^0,x) - J^{0,N}_{\phi^0,\phi}(t,i^0,x)| = |v^{N}_\epsilon(t,i^0,x) - v^{N}(t,i^0,x)|\le L\epsilon.
\]
The desired results follow by combining the above inequalities and choosing $\epsilon = 1/\sqrt{N}$.
\end{proof}

\subsection{Approximative Nash Equilibria}
The following is  an immediate consequence of the above propagation of chaos results.

\begin{theorem}
Assume that the Mean Field Game attains Nash equilibrium when the major player chooses a Lipschitz strategy $\hat\alpha$ and all the minor players choose a Lipschitz strategy $\hat\beta$. Denote $L_0$ the Lipschitz constant for $\hat\alpha$ and $\hat\beta$. Then for any $L\ge L_0$, $(\hat\alpha,\hat\beta)$ is an approximative Nash equilibrium within all the $L-$Lipschitz strategies. More specifically, there exist constants $C>0$ and $N_0\in\mathbb{N}$ depending on $L_0$ such that for all $N\ge N_0$, all strategy $\alpha$ for major player and $\beta$ for minor player that are $L$-Lipschitz, we have
\begin{align*}
J^{0,N}(\hat\alpha, \hat\beta,\dots, \hat\beta) \ge\;\;& J^{0,N}(\alpha, \hat\beta,\dots, \hat\beta) - C/\sqrt{N}\\
J^{N}(\hat\alpha, \hat\beta,\dots, \hat\beta) \ge\;\;& J^{N}(\hat\alpha, \hat\beta,\dots, \beta, \dots \hat\beta)- C/\sqrt{N}
\end{align*}
\end{theorem}

\section{Appendix}
\label{se:appendix}

\subsection{Proof of Lemma \ref{dynamicprogoperator}}
Recall the dynamic programming operator defined in (\ref{dynamicprogoperatordef}). Let $\Phi$ be a mapping on $[0,T]\times\{1,\dots,M^0\}\times\mathcal{P}$ such that $\Phi(\cdot, i^0, \cdot)$ is $\mathcal{C}^1$ in $[0,T]\times\mathcal{P}$ and $\Phi(\cdot, j^0, \cdot)$ is continuous in $[0,T]\times\mathcal{P}$ for $j^0\neq i^0$. We are going to evaluate the following limit:
\[
\lim_{h\rightarrow 0}\frac{1}{h}\left\{[\mathcal{T}_{t,t+h} \Phi(t+h, \cdot,\cdot)](i^0, x) - \Phi(t,i^0,x)\right\}:=\lim_{h\rightarrow 0}I_h
\]
(i) Let us first assume that $\Phi(\cdot, j^0, \cdot)$ is $\mathcal{C}^1$ in $[0,T]\times\mathcal{P}$ for all $j^0$. Consider a constant control $\alpha^0$, then by definition of the operator we have:
\[
I_h\le \frac{1}{h}\left\{\mathbb{E}\left[\int_{t}^{t+h}f^0(u, X_u^0, \alpha^0, \mu_u)du + \Phi(t+h, X_{t+h}^0, \mu_{t+h})|X_t^0=i^0, \mu = x \right] - \Phi(t, i^0, x)\right\}
\]
Using the infinitesimal generator of the process $(X_u^0, \mu_u)$, the RHS of the above inequality has a limit. Take the limit and take the infimum over $\alpha$, we obtain:
\begin{align*}
\limsup_{h\rightarrow 0} I_h\le&\;\partial_t \Phi(t,i^0, x) + \inf_{\alpha^0\in A^0}\bigg\{f^0(t, i^0, \alpha^0, x) + (1-\sum_{k=1}^{M-1} x_k)  \sum_{k=1}^{M-1} \partial_{x_k} \Phi(t,i^0, x)  q(t,M, k, \phi(t,M, i^0, x), i^0, \alpha^0, x)\\
&+\;\sum_{i,j=1}^{M-1} \partial_{x_j} \Phi(t,i^0, x)  x_i  q(t,i,j,\phi(t,i, i^0,x), i^0, \alpha^0, x)+ \sum_{j^0\neq i^0} [\Phi(t,j^0, x) - \Phi(t,i^0, x)] q^0(i^0, j^0, \alpha^0, x)\bigg\}
\end{align*}
On the other hand, for all $h>0$, there exists a control $\phi^{0}_h$ such that
\[
\mathbb{E}\left[\int_{t}^{t+h}f^0(u, X_u^0, \phi_h^0(u, X_u^0, \mu_u), \mu_u)du + \Phi(t+h, X_{t+h}^0, \mu_{t+h})|X_t^0=i^0, \mu = x \right] \le \mathcal{T}_{t,t+h} \Phi(t+h, \cdot,\cdot))(i^0, x) + h^2
\]
This implies:
\[
I_h \ge  \frac{1}{h}\left\{\mathbb{E}\left[\int_{t}^{t+h}f^0(u, X_u^0, \phi^0_h(u, X_u^0, \mu_u), \mu_u)du + \Phi(t+h, X_{t+h}^0, \mu_{t+h})|X_t^0=i^0, \mu = x \right] - \Phi(t, i^0, x)\right\} - h
\]
Since $\Phi(\cdot, j^0, \cdot)$ is $\mathcal{C}^1$ for all $j^0$, this can be further written using the infinitesimal generator:
\begin{align*}
I_h \ge\;\;&\;  \frac{1}{h}\mathbb{E}\left[\int_{t}^{t+h}f^0(u, X_u^0, \alpha^h(u, X_u^0, \mu_u), \mu_u) + [\mathcal{G}^0_{\phi^0_h, \phi}\Phi](u, X_u^0, \mu_u) du|X_t^0=i^0, \mu = x \right] -h
\end{align*}
Taking supremum over the control and applying Dominated Convergence Theorem, we obtain:
\begin{align*}
\;\liminf_{h\rightarrow 0} I_h\ge&\;\partial_t \Phi(t,i^0, x) + \inf_{\alpha^0\in A^0}\bigg\{f^0(t, i^0, \alpha^0, x) + (1-\sum_{k=1}^{M-1} x_k)  \sum_{k=1}^{M-1} \partial_{x_k} \Phi(t,i^0, x)  q(t,M, k, \phi(t,M, i^0, x), i^0, \alpha^0, x)\\
&+\;\sum_{i,j=1}^{M-1} \partial_{x_j} \Phi(t,i^0, x)  x_i  q(t,i,j,\phi(t,i, i^0,x), i^0, \alpha^0, x)+ \sum_{j^0\neq i^0} [\Phi(t,j^0, x) - \Phi(t,i^0, x)] q^0(i^0, j^0, \alpha^0, x)\bigg\}
\end{align*}
This proves the lemma for $\Phi$ such that $\Phi(\cdot, j^0, \cdot)$ is $\mathcal{C}^1$ for all $j^0$.

\vspace{3mm}
\noindent(ii) Now take a continuous mapping $\Phi$ and only assume that $\Phi(\cdot, i^0, \cdot)$ is $\mathcal{C}^1$. Applying Weierstrass approximation theorem, for any $\epsilon>0$, there exists $\mathcal{C}^1$ function $\phi^{\epsilon}_{j^0}$ on $[0,T]\times\mathcal{P}$ for all $j^0\neq i^0$ such that
\[
\sup_{(t,x)\in[0,T]\times\mathcal{P}} | \Phi(t,j^0,x) - \phi^{\epsilon}_{j^0}(t,x)| \le \epsilon
\]
Define $\Phi^{\epsilon}(t,j^0,x):=\phi^{\epsilon}_{j^0}(t,x) + \epsilon$ for $j^0\neq i^0$ and $\Phi^{\epsilon}(t,i^0,x):= \Phi(t,i^0,x)$. Then we have $\Phi^{\epsilon} \ge \Phi$. By the monotonicity of the operator $\mathcal{T}_{t,t+h}$ we have:
\[
\frac{1}{h}\left\{[\mathcal{T}_{t,t+h} \Phi^\epsilon(t+h, \cdot,\cdot)](i^0, x) - \Phi^\epsilon(t,i^0,x)\right\}\ge\frac{1}{h}\left\{[\mathcal{T}_{t,t+h} \Phi(t+h, \cdot,\cdot)](i^0, x) - \Phi(t,i^0,x)\right\} :=I_h
\]
Now apply the results from step (i), we obtain
{\allowdisplaybreaks
\begin{align*}
&\;\limsup_{h\rightarrow 0}I_h \\
\le&\; \partial_t \Phi^\epsilon(t,i^0, x) + \inf_{\alpha^0\in A^0}\bigg\{f^0(t, i^0, \alpha^0, x) + (1-\sum_{k=1}^{M-1} x_k)  \sum_{k=1}^{M-1} \partial_{x_k} \Phi^\epsilon(t,i^0, x)  q(t,M, k, \phi(t,M, i^0, x), i^0, \alpha^0, x)\\
&+\;\sum_{i,j=1}^{M-1} \partial_{x_j} \Phi^\epsilon(t,i^0, x)  x_i  q(t,i,j,\phi(t,i, i^0,x), i^0, \alpha^0, x)+ \sum_{j^0\neq i^0} [\Phi^\epsilon(t,j^0, x) - \Phi^\epsilon(t,i^0, x)] q^0(t,i^0, j^0, \alpha^0, x)\bigg\}\\
=&\; \partial_t \Phi(t,i^0, x) + \inf_{\alpha^0\in A^0}\bigg\{f^0(t, i^0, \alpha^0, x) + (1-\sum_{k=1}^{M-1} x_k)  \sum_{k=1}^{M-1} \partial_{x_k} \Phi(t,i^0, x)  q(t,M, k, \phi(t,M, i^0, x), i^0, \alpha^0, x)\\
&+\;\sum_{i,j=1}^{M-1} \partial_{x_j} \Phi(t,i^0, x)  x_i  q(t,i,j,\phi(t,i, i^0,x), i^0, \alpha^0, x)\\
&+\;\sum_{j^0\neq i^0} [\Phi^\epsilon(t,j^0, x) - \Phi(t,j^0, x) + \Phi(t,j^0, x) - \Phi(t,i^0, x)] q^0(t,i^0, j^0, \alpha^0, x)\bigg\}\\
\le&\; \partial_t \Phi(t,i^0, x) + \inf_{\alpha^0\in A^0}\bigg\{f^0(t, i^0, \alpha^0, x) + (1-\sum_{k=1}^{M-1} x_k)  \sum_{k=1}^{M-1} \partial_{x_k} \Phi(t,i^0, x)  q(t,M, k, \phi(t,M, i^0, x), i^0, \alpha^0, x)\\
&+\;\sum_{i,j=1}^{M-1} \partial_{x_j} \Phi(t,i^0, x)  x_i  q(t,i,j,\phi(t,i, i^0,x), i^0, \alpha^0, x)+ \sum_{j^0\neq i^0} [\Phi(t,j^0, x) - \Phi(t,i^0, x)] q^0(t,i^0, j^0, \alpha^0, x)\bigg\} + \epsilon  L
\end{align*}
}
The last equality is due to the fact that $q^0$ is bounded and $|\Phi^\epsilon(t,j^0, x) - \Phi(t,j^0, x)|\le \epsilon$. We can write a similar inequality for $\liminf_{h\rightarrow 0}I_h$. Then tending $\epsilon$ to $0$ yields the desired result.

\subsection{Proof of Theorem \ref{compmajor}}
In this section we present the proof of the comparison principle for HJB equation associated with major player's optimization problem. The arguments used in this proof can be readily applied to prove uniqueness of solution to minor player's HJB equation (\ref{hjbminor}) (c.f. Theorem \ref{hjbtheominor}). The same argument can also be used to prove the uniqueness result for equations (\ref{majorJpde}) and (\ref{minorJpde}) (c.f. Proposition \ref{pdecharacterizepayoff}).

Let $v$ and $w$ be respectively viscosity subsolution and supersolution to equation (\ref{hjbmajor}). Our objective is to show $v(t, i, x) \le w(t,i,x)$ for all $1\le i\le M^0$, $x\in\mathcal{P}$ and $t\in[0,T]$.

\vspace{3mm}
\noindent (i) Without loss of generality, we may assume that $v$ is a viscosity subsolution of:
\begin{equation}\label{hjbmajoreta}
\begin{aligned}
0=\;\;&\; -\eta + \partial_t v^0(t, i^0, x) + \inf_{\alpha^0 \in A^0}\bigg\{f^0(t, i^0, \alpha^0, x) + \sum_{j^0\neq i^0} [v^0(t,j^0, x) - v^0(t,i^0, x)] q^0(t,i^0, j^0, \alpha^0, x)\\
&\;+(1-\sum_{k=1}^{M-1} x_k)  \sum_{k=1}^{M-1} \partial_{x_k} v^0(t,i^0,x)  q(t,M, k, \phi(t,M, i^0, x), i^0, \alpha^0, x)\\
&\;+\sum_{i,j=1}^{M-1} \partial_{x_j} v^0(t, i^0, x)  x_i  q(t,i,j,\phi(t,i,i^0,x), i^0,\alpha^0, x)\bigg\},\;\;\;\forall (i_0,t,x)\in \{1,\dots, M^0\} \times [0,T[ \times \mathcal{P}\\
0=\;\;&v^0(T,i^0, x) - g^0(i^0, x),\;\;\;\forall x\in \mathcal{P}\\
\end{aligned}
\end{equation}
where $\eta>0$ is a small parameter. Indeed we may consider the function $v_{\eta}(t,i,x):= v(t,i,x) - \eta(T-t)$. Then it is easy to see that $v_\eta$ is a viscosity subsolution to the above equation. If we can prove $v_\eta \le w$, the tending $\eta$ to $0$ yields $v \le w$. In the following, we will only consider the subsolution $v$ to equation (\ref{hjbmajoreta}) and the supersolution $w$ to the equation (\ref{hjbmajor}), and try to prove $v\le w$.

\vspace{3mm}
\noindent(ii) For $\epsilon>0$ and $1\le i^0 \le M^0$, consider the function $\Gamma_{i^0,\epsilon}$ defined on $[0,T]^2\times \mathcal{P}^2$:
\[
\Gamma_{i^0,\epsilon}(t,s,x,y) := v(t,i^0,x) - w(s,i^0,y) - \frac{1}{\epsilon}|t-s|^2 - \frac{1}{\epsilon}\|x-y\|^2
\]
where $\|\cdot\|$ is the euclidian norm on $\mathbb{R}^{(M^0 - 1)}$. Since $\Gamma_{i^0,\epsilon}$ is a continuous function on a compact set, it attains the maximum denote as $N_{i^0,\epsilon}$. Denote $(\bar t, \bar s, \bar x, \bar y)$ the maximizer (which obviously depends on $\epsilon$ and $i^0$, but for simplicity of the notation we suppress the notation). We show that for all $1\le i^0 \le M^0$, there exists a sequence $\epsilon_n \rightarrow 0$ and the corresponding maximizer $(\bar t_n,\bar s_n,\bar x_n, \bar y_n)$ such that
\begin{subequations}
\begin{equation}\label{seq1}
 (\bar t_n,\bar s_n,\bar x_n, \bar y_n) \rightarrow (\hat t, \hat t, \hat x, \hat x), \text{where}\;\;
(\hat t,\hat x) := \arg\sup_{(t,x)\in[0,T]\times\mathcal{P}}\{v(t,i^0,x)- w(t,i^0,x)\}
\end{equation}
\begin{equation}\label{seq2}
\frac{1}{\epsilon_n}|\bar t_n - \bar s_n|^2 + \frac{1}{\epsilon_n}\|\bar x_n - \bar y_n\|^2 \rightarrow 0
\end{equation}
\begin{equation}\label{seq3}
N_{i^0,\epsilon}\rightarrow N_{i^0} := \sup_{(t,x)\in[0,T]\times\mathcal{P}}\{v(t,i^0,x)- w(t,i^0,x)\}
\end{equation}
\end{subequations}
Indeed for any $(t,x)\in[0,T]\times\mathcal{P}$, we have $v(t,i^0,x)-w(t,i^0,x) = \Gamma_{i^0,\epsilon}(t,t,x,x) \le N_{i^0,\epsilon}$. Taking the supremum we obtain $N_{i^0}\le N_{i^0,\epsilon}$ and therefore
\[
\frac{1}{\epsilon}|\bar t - \bar s|^2 + \frac{1}{\epsilon}\|\bar x - \bar y\|^2 \le v(\bar t, i^0, \bar x) - w(\bar s, i^0, \bar y) - N_{i^0} \le 2L - N_{i^0}
\]
The last inequality comes from the fact that $v$ and $w$ are bounded on the compact $[0,T]\times\mathcal{P}$. It follows that $|\bar t - \bar s|^2 + \|\bar x - \bar y\|^2 \rightarrow 0$. Now since the sequence $(\bar t, \bar s, \bar x, \bar y)$ (indexed by $\epsilon$) is in the compact, we can extract a subsequence $\epsilon_n\rightarrow 0$ such that $(\bar t_n,\bar s_n,\bar x_n, \bar y_n) \rightarrow (\hat t, \hat t, \hat x, \hat x)$. We have the following inequality:
\[
N_i \le v(\bar t_n, i^0, \bar x_n) - w(\bar s_n, i^0, \bar y_n) - \frac{1}{\epsilon_n}|\bar t_n - \bar s_n|^2 - \frac{1}{\epsilon_n}\|\bar x_n - \bar y_n\|^2 = N_{i^0,\epsilon_n} \le v(\bar t_n, i^0, \bar x_n) - w(\bar s_n, i^0, \bar y_n)
\]
Notice that $v(\bar t_n, i^0, \bar x_n) - w(\bar s_n, i^0, \bar y_n)\rightarrow v(\hat t, i^0, \hat x_n) - w(\hat t, i^0, \hat x_n) \le N_{i^0}$, we deduce that $N_{i^0} = v(\hat t, i^0, \hat x_n) - w(\hat t, i^0, \hat x_n)$ which implies (\ref{seq1}). (\ref{seq2}) and (\ref{seq3}) follows easily by taking the limit in the above inequality.

\vspace{3mm}
\noindent (iii) Now we prove the comparison principle. Using the notation introduced in step (ii), we need to prove $N_{i^0}\le 0$ for all $1\le i^0\le M^0$. Assume that there exists $1 \le i^0 \le M^0$ such that
\[
N_{i^0} = \sup_{1\le j^0 \le M^0} N_{j^0} >0
\]
We work towards a contradiction. Without loss of generality we assume that $N_{i^0} > N_{j^0}$ for all $j\neq i$. We then consider the subsequence $(\epsilon_n, \bar t_n, \bar s_n, \bar x_n, \bar y_n) \rightarrow (0, \hat t, \hat t, \hat x,\hat x)$ with regard to $i^0$ constructed in step (ii), for which (\ref{seq1}), (\ref{seq2}) and (\ref{seq3}) are satisfied. Since $v(\hat t, i^0, \hat x) - w(\hat t, i^0, \hat x) = N_{i^0} > 0$, we have $\hat t\neq T$. Instead of extracting a subsequence, we may assume that $\bar t_n\neq T$ and $\bar s_n \neq T$ for all $n\ge 0$. Moreover for any $j^0\neq i^0$, we have
\[
v(\hat t, j^0, \hat x) - w(\hat t, j^0, \hat x) \le N_j < N_i = v(\hat t, i^0, \hat x) - w(\hat t, i^0, \hat x)
\]
Since $v(\bar t_n, j^0, \bar x_n) - w(\bar s_n, j^0, \bar y_n) \rightarrow v(\hat t, j^0, \hat x) - w(\hat t, j^0, \hat x)$, instead of extracting a subsequence, we can assume that for all $j^0\neq i^0$ and $n\ge 0$,
\begin{equation}\label{seq4}
v(\bar t_n, j^0, \bar x_n) - w(\bar s_n, j^0, \bar y_n) \le v(\bar t_n, i^0, \bar x_n) - w(\bar s_n, i^0, \bar y_n)
\end{equation}
In the following we suppress the index $n$ for the sequence $(\epsilon_n, \bar t_n, \bar s_n, \bar x_n, \bar y_n)$ for sake of simplicity of notation. By definition of the maximizer, for any $(t,x)\in[0,T]\times\mathcal{P}$, we have:
\[
v(t,i^0,x) - w(\bar s,i^0, \bar y) - \frac{1}{\epsilon}\|x - \bar y\|^2- \frac{1}{\epsilon}|t - \bar s|^2 \le v(\bar t,i^0,\bar x) - w(\bar s,i^0, \bar y) - \frac{1}{\epsilon}\|\bar x - \bar y\|^2- \frac{1}{\epsilon}|\bar t - \bar s|^2
\]
Therefore $v(\cdot,i^0,\cdot) - \phi$ attains maximum at $(\bar t,\bar x)$ where
\begin{equation}\label{constsub1}
\phi(t,x) := \frac{1}{\epsilon}\|x - \bar y\|^2 + \frac{1}{\epsilon}|t - \bar s|^2\;\;\;\;\partial_t\phi(\bar t, \bar x) = \frac{2}{\epsilon}(\bar t - \bar s)\;\;\;\;
\nabla\phi(\bar t, \bar x) = \frac{2}{\epsilon}(\bar x - \bar y)
\end{equation}
Similarly $w(\cdot,i^0,\cdot) - \psi$ attains minimum at $(\bar s,\bar y)$ where
\begin{equation}\label{constsub2}
\psi(t,x) := -\frac{1}{\epsilon}\|x - \bar x\|^2 - \frac{1}{\epsilon}|t - \bar t|^2\;\;\;\;
\partial_t\psi(\bar s, \bar y) = \frac{2}{\epsilon}(\bar t - \bar s)\;\;\;\;
\nabla\psi(\bar s,\bar y) = \frac{2}{\epsilon}(\bar x - \bar y)
\end{equation}
Since $\bar s\neq T$ and $\bar t \neq T$, the definition of viscosity solution and (\ref{constsub1}), (\ref{constsub2}) gives the following inequalities:
\begin{align*}
\eta\le\;\;&\;\frac{2}{\epsilon}(\bar t - \bar s) + \inf_{\alpha^0\in A^0}\bigg\{f^0(\bar t, i^0, \alpha^0, \bar x) +(1-\sum_{k=1}^{M-1} \bar x_k)  \sum_{k=1}^{M-1}\frac{2}{\epsilon}(\bar x_k - \bar y_k)  q(\bar t, M, k, \phi(\bar t,M, i^0, \bar x), i^0, \alpha^0, \bar x)\\
&+\;\sum_{i,j=1}^{M-1} \frac{2}{\epsilon}(\bar x_j - \bar y_j)  \bar x_i  q(\bar t, i,j,\phi(\bar t,i,i^0,\bar x), i^0, \alpha^0, \bar x)+ \sum_{j^0\neq i^0} [v(\bar t,j^0, \bar x) - v(\bar t,i^0, \bar x)] q^0(\bar t, i^0, j^0, \alpha^0, \bar x)\bigg\}\\
0\ge\;\;&\;\frac{2}{\epsilon}(\bar t - \bar s) + \inf_{\alpha^0\in A^0}\{f^0(\bar s, i^0, \alpha^0, \bar y) + (1-\sum_{k=1}^{M-1} \bar y_k)  \sum_{k=1}^{M-1} \frac{2}{\epsilon}(\bar x_k - \bar y_k)  q(\bar s, M, k, \phi(\bar s,M, i^0, \bar y), i^0, \alpha^0, \bar y)\\
&+\; \sum_{i,j=1}^{M-1} \frac{2}{\epsilon}(\bar x_j - \bar y_j)  \bar y_i  q(\bar s, i,j,\phi(\bar s,i,i^0, \bar y), i^0, \alpha^0, \bar y)+ \sum_{j^0\neq i^0} [w(\bar s, j^0, \bar y) - w(\bar s,i^0, \bar y)] q^0(\bar s, i^0, j^0, \alpha^0, \bar y)\bigg\}
\end{align*}
Substracting the above two inequalities, we obtain:
\begin{equation}\label{contra}
0 < \eta \le I_1 + I_2 + I_3
\end{equation}
where the three terms $I_1$, $I_2$ and $I_3$ will be dealt with in the following. For $I_1$ we have:
\begin{align*}
I_1:=\;\;&\;\sup_{\alpha^0\in A^0}\{f^0(\bar t, i^0, \alpha^0, \bar x)-f^0(\bar s, i^0, \alpha^0, \bar y)\} \\
&\;+ \sup_{\alpha^0\in A^0}\{\sum_{j^0\neq i^0} [v(\bar t,j^0, \bar x) - v(\bar t,i^0, \bar x)] q^0(\bar t, i^0, j^0, \alpha^0, \bar x) - \sum_{j^0\neq i^0} [w(\bar s, j^0, \bar y) - w(\bar s,i^0, \bar y)] q^0(\bar s, i^0, j^0, \alpha^0, \bar y)\}\\
\le\;\;&\;  \sup_{\alpha^0\in A^0}\{f^0(\bar t, i^0, \alpha^0, \bar x) - f^0(\bar s, i^0, \alpha^0, \bar y)\} \\
&\;+ \sup_{\alpha^0\in A^0}\{\sum_{j^0\neq i^0}[v(\bar t,j^0, \bar x) - v(\bar t,i^0, \bar x)] (q^0(\bar t, i^0, j^0, \alpha^0, \bar x) - q^0(\bar s, i^0, j^0, \alpha^0, \bar y))\}\\
&\; + \sup_{\alpha^0\in A^0}\{\sum_{j^0\neq i^0}[v(\bar t,j^0, \bar x) - w(\bar s, j^0, \bar y) - v(\bar t,i^0, \bar x) + w(\bar s,i^0, \bar y)] q^0(\bar s, i^0, j^0, \alpha^0, \bar y)\}\\
\le\;\;& L(|\bar t - \bar s| + \|\bar x - \bar y\|) + 2C  L(|\bar t - \bar s| + \|\bar x - \bar y\|) + 0
\end{align*}
In the last inequality we use (\ref{seq4}) and the fact that $q^0(\bar s, i^0, j^0, \alpha^0, \bar y) \ge 0$ for $j^0\neq i^0$. We also use the Lipschitz property of $f^0$ and $q^0$. Now in light of $(\ref{seq1})$, we obtain $I_1\rightarrow 0$ as $\epsilon\rightarrow 0$. Now turning to $I_2$:
\begin{align*}
I_2:=\;\;&\; \sup_{\alpha^0\in A^0}\bigg\{\sum_{i,j=1}^{M-1}\frac{2}{\epsilon}(\bar x_j - \bar y_j)  \left [ \bar x_i  q(\bar t, i,j,\phi(\bar t,i,i^0,\bar x), i^0,\alpha^0, \bar x) - \bar y_i  q(\bar s, i,j,\phi(\bar s,i,i^0, \bar y), i^0,\alpha^0, \bar y)\right]\bigg\}\\
\le\;\;&\;\sup_{\alpha^0\in A^0}\bigg\{ \sum_{i,j=1}^{M-1}\frac{2}{\epsilon}|(\bar x_j - \bar y_j)(\bar x_i - \bar y_i)q(\bar t, i,j,\phi(\bar t,i,i^0,\bar x), i^0, \alpha^0,\bar x) |\bigg\} \\
&\;+ \sup_{\alpha^0\in A^0}\bigg\{\sum_{i,j=1}^{M-1}\frac{2}{\epsilon}|\bar y_i (\bar x_j - \bar y_j)(q(\bar t, i,j,\phi(\bar t,i,i^0,\alpha^0,\bar x), i^0, \alpha^0,\bar x) - q(\bar s, i,j,\phi(\bar s,i,i^0, \bar y), i^0, \alpha^0, \bar y))|\bigg\}\\
\le\;\;& \; \sum_{i,j=1}^{M-1}\frac{2}{\epsilon}C |\bar x_j - \bar y_j|  | \bar x_i - \bar y_i| + \sum_{j=1}^{M-1}\frac{2}{\epsilon}|\bar x_j - \bar y_j| L(M-1)(|\bar t - \bar s| + \|\bar x - \bar y\|)
\end{align*}
where in the last inequality we used the Lipschitz property of $q$ and $\phi$ uniformly in $\alpha$, as well as the boundedness of the function $q$. It follows that $I_2 \le C  \frac{1}{\epsilon}(|\bar t - \bar s|^2 + \|\bar x - \bar y\|^2)$ and by (\ref{seq2}) we see that $I_2\rightarrow 0$ as $\epsilon\rightarrow 0$. Finally we deal with $I_3$, which is defined by:
\begin{align*}
I_3:=\;\;&\; \sup_{\alpha^0\in A^0}\bigg\{(1-\sum_{k=1}^{M-1} \bar x_k)  \sum_{k=1}^{M-1}\frac{2}{\epsilon}(\bar x_k - \bar y_k)  q(\bar t, M, k, \phi(\bar t,M, i^0, \bar x), i^0, \alpha^0, \bar x) \\
&\;- (1-\sum_{k=1}^{M-1} \bar y_k)  \sum_{k=1}^{M-1} \frac{2}{\epsilon}(\bar x_k - \bar y_k)  q(\bar s, M, k, \phi(\bar s,M, i^0, \bar y), i^0,\alpha^0,  \bar y)\bigg\}
\end{align*}
Using a similar estimation as for $I_2$, we obtain $I_3 \le C  \frac{1}{\epsilon}(|\bar t - \bar s|^2 + \|\bar x - \bar y\|^2)$. Therefore by tending $\epsilon$ to $0$ in the inequality (\ref{contra}), we obtain a contradiction. This completes the proof.

\subsection{Proof of Lemma \ref{lemexistencenash}}
The main tool we use for the proof is the theory of limit operation on viscosity solution. We refer the reader to \cite{lions} for an introductory presentation of limit operation on viscosity solution to non-linear second order PDE. Here we adapt the results established therein to the case of coupled system of non-linear first order PDE, which is the HJB equation we are interested in throughout the paper. Let $\mathcal{O}$ be some locally compact subset of $\mathbb{R}^d$ and $(F_n)_{n\ge 0}$ be a sequence of functions defined on $\{1,\dots,M\}\times\mathcal{O}$. We define the limit operator $\limsup^*$ and $\liminf_*$ as follow:
\begin{align*}
{\limsup}^* F_n(i, x) :=& \lim_{n\rightarrow +\infty, \epsilon\rightarrow 0} \sup\{F_k(i, y) | k\ge n, \|y - x\|\le \epsilon\}\\
{\liminf}_* F_n(i, x) :=& \lim_{n\rightarrow +\infty, \epsilon\rightarrow 0} \inf\{F_k(i, y) | k\ge n, \|y - x\|\le \epsilon\}
\end{align*}
Intuitively, we expect that the limit of a sequence of viscosity solutions solves the limiting PDE. It turns out that this is the proper definition of the limit operation, as is stated in the following lemma:
\begin{lemma}\label{lemprooflem1}
Let $\mathcal{O}$ be a locally compact subset of $\mathbb{R}^d$, $(u_n)_{n\ge 0}$ be a sequence of continuous functions defined on $\{1,\dots, M\}\times\mathcal{O}$ and $H_n$ be a sequence of functions defined on $\{1, \dots, M\} \times [0,T] \times \mathcal{O} \times \mathbb{R}^M \times \mathbb{R}^d$, such that for each $n$, $u_n$ is viscosity solutions to the system of PDEs $H_n(i, t, x, u, \partial_t u(\cdot,i,\cdot), \nabla u(\cdot,i,\cdot)) = 0, \forall 1\le i\le M$, in the sense of Definition \ref{viscositydefi}. Assume that for each $i,x,d,p$, the mapping $u \rightarrow H(i,t,x,u,d,p)$ is non-decreasing in the $j$-th component of $u$, for all $j\neq i$. Then $u^* := \limsup^*u_n$ (resp. $u_* := \liminf_* u_n$) is viscosity subsolution (resp. supersolution) to $H^*(i, t, x, u, \partial_t u(i,\cdot), \nabla u(i,\cdot)) = 0, \forall 1\le i\le M$ (resp. $H_*(i, t, x, u,\partial_t u(\cdot,\cdot, i,\cdot), \nabla u(\cdot,\cdot, i,\cdot)) = 0, \forall 1\le i\le M$).
\end{lemma}
The proof requires the definition of viscosity subsolution and supersolution based on the notion of first-order subjet and superjet (see definition 2.2 in \cite{lions}). Let $u$ be a continuous function defined on $[0,T]\times\{1,\dots, M\}\times\mathcal{O}$, we define the first order superjet of $u$ on $(t,i,x)$ to be:
\[
\mathcal{J}^+ u (t,i,x) := \{(d,p)\in\mathbb{R}\times\mathbb{R}^d : u(s,i,y) \le u(t,i,x) + d(s-t) + (y-x) p + o(|t-s| + \|y-x\|), (s,y)\rightarrow(t,x) \}
\]
Then $u$ is a viscosity subsolution (resp. supersolution) to the system $H(i, t, x, u, \partial_t u(i,\cdot), \nabla u(i,\cdot))=0$ if only if for all $(i,t,x)$ and $(d,p)\in \mathcal{J}^+ u (t,i,x)$ (resp. $(d,p)\in \mathcal{J}^- u (t,i,x)$), we have $H(i, t, x, u, d, p) \ge 0$ (resp. $H(i, t, x, u, d, p) \le 0$).

\begin{proof}
Fix $(t,i,x)\in [0,T]\times \{1,\dots,M\}\times\mathcal{O}$ and let $(d,p)\in\mathcal{J}^+ u^*(t,i,x)$. We want to show that:
\[
H^*(i,t,x,u^*(t,\cdot,x), d, p) \ge 0
\]
where $u^*(t, \cdot , x)$ represents the $M-$dimensional vector $[u^*(t, k , x)]_{1\le k\le M}$.
By lemma 6.1., there exists a sequence $n_j\rightarrow+\infty$, $x_j\in\mathcal{O}$, $(d_j, p_j)\in \mathcal{J}^+ u_{n_j}(t_j, i, x_j)$ such that:
\[
(t_j, x_j, u_{n_j}(t_j, i, x_j), d_j, p_j) \rightarrow(t,x,u^*(t, i, x), d, p), \;\;\;j\rightarrow +\infty
\]
Since $u_{n}$ is viscosity subsolution to $H_n = 0$, we have
\[
H_{n_j}(i, t_j, x_j, u_{n_j}(t_j, \cdot , x_j), d_j, p_j) \ge 0
\]
Now let us denote
\[
S^u_{n, \epsilon}(t,k,x):= \sup\{u_j(s,k,y), \max(|s-t|, \|y-x\|)\le\epsilon, j\ge n\}
\]
Then we have $u^*(t,i,x) = \lim_{n\rightarrow +\infty, \epsilon\rightarrow 0} S^u_{n,\epsilon}(t,i,x)$. Fix $\delta >0$, then there exists $\epsilon^0>0$ and $N^0>0$ such that for all $\epsilon\le  \epsilon^0$ and $j\ge N^0$, we have 
\[
|S^u_{n_j, \epsilon}(t,k,x) - u^*(t,k,x)| \le \delta,\;\;\;\forall k\neq i
\]
 Moreover, there exists $N>0$ such that for all $j\ge N$, 
\[
\|(t_j, x_j, u_{n_j}(t_j, i, x_j), d_j, p_j) - (t,x,u^*(t,i,x),d,p)\|\le \delta\wedge \epsilon^0
\]
Then for any $j\ge N^0\vee N$, we have $\|(t_j, x_j) - (t,x)\|\le \epsilon^0$, and by definition of $S^u_{\epsilon,n}$ we deduce that $u_{n_j}(t_j, k, x_j) \le S^u_{n_j,\epsilon^0}(t,k,x)$ for all $k\neq i$. By the monotonicity property of $H_{n_j}$, we have:
\[
H_{n_j}(i, t_j, x_j, S^u_{n_j,\epsilon^0}(t,1,x), \dots, S^u_{n_j,\epsilon^0}(t,i-1,x), u_{n_j}(t_j, i, x_j), S^u_{n_j,\epsilon^0}(t,i +1,x),\dots, S^u_{n_j,\epsilon^0}(t,M,x),   d_j, p_j) \ge 0
\]
Now in the above inequality, all the arguments of $H_{n_j}$ except $i$ is located in a ball of radius $\delta$ centered on the point $(t,x,u^*(t,\cdot,x), d, p)$. We have thus
\[
S^H_{n_j,\delta}(i,t,x,u^*(t,\cdot,x), d, p) \ge 0
\]
where we have defined:
\[
S^H_{n,\delta}(i,t,x,u, d, p):=\sup\{H_{j}(i,s,y, v, e, q), \|(s,y,v,e,q) - (t,x,u,d,p)\|\le\epsilon, j\ge n\}
\]
We have just proved that for any $\delta >0$, there exists $M>0$ such that for any $j\ge M$, we have $S^H_{n_j,\delta}(i,t,x,u^*(t,\cdot,x), d, p) \ge 0$. Since we have
\[
H^*(i,t,x,,u^*(t,\cdot,x), d, p) = \lim_{n\rightarrow +\infty, \epsilon \rightarrow 0} S^H_{n_j,\delta}(i,t,x,u^*(t,\cdot,x), d, p)
\]
We deduce that $H^*(i,t,x,,u^*(t,\cdot,x), d, p)\ge 0$.
\end{proof}
Now going back to the proof of Lemma \ref{lemexistencenash}, we consider a converging sequence of elements $(\phi^0_n, \phi_n) \rightarrow (\phi^0, \phi)$ in $\mathbb{K}$, where we have defined $\mathbb{K}$ to be the collection of  major and minor player's controls $(\phi^0, \phi)$ that are $L-$Lipschitz in $(t,x)$. We denote $V^0_n$ (resp. $V^0$) the value function of major player's control problem associated with the controls $\phi_n$ (resp. $\phi$). We also use the notation $V^{0*} := \limsup V_n^0$ and $V^{0}_* := \liminf V_n^0$. For all $n\ge 0$, we define the operator $H_n$:
\begin{align*}
&H^0_n(i^0,t,x,u,d,p) :=\;\;d + \inf_{\alpha^0\in A^0} \bigg\{f^0(t,i^0,\alpha^0, x) + \sum_{j^0\neq i^0} (u_{j^0} - u_{i^0}) q^0(t,i^0,j^0, \alpha^0,x) \\
&\;\;+ (1-\sum_{k=1}^{M}x_k)  \sum_{k=1}^{M-1}p_k  q(t,M,k,\phi_n(t, M, i^0, x), i^0,\alpha^0, x)+ \sum_{k,l=1}^{M-1}p_l x_k q(t,k,l,\phi_n(t,k,i^0,x),i^0,\alpha^0,x)\bigg\}, \;\;\;\text{if}\;\;\; t<T \\
&H^0_n(i^0,T,x,u,d,p) :=\;\; g^0(i^0,x) - u_{i^0}
\end{align*}
Then $V^0_n$ is viscosity solution to the equation $H^0_n = 0$. It is clear to see that the operator $H_n^0$ satisfies the monotonicity condition in Lemma \ref{lemprooflem1}. To evaluate $H^{0*}:=\limsup H_n^0$ and $H^0_*:=\liminf H_n^0$, we remark that for each $1\le i^0\le M^0$, the sequence of functions $(t,x,u,d,p)\rightarrow H_n(i^0,t,x,u,d,p)$ is equicontinuous. Indeed, this is due to the fact that the sequence $\phi_n$ is equicontinuous and the function $q$ is Lipschitz. Therefore $H^{0*}$ and $H^0_*$ are simply the limit in the pointwise sense when $t<T$. When $t=T$, the boundary condition needs to be taken into account. The following computation is straightforward to verify:
\begin{lemma}\label{lemprooflem2}
Define the operator $H^0$ as:
\begin{align*}
&H^0(i^0,t,x,u,d,p) :=\;\;d + \inf_{\alpha^0\in A^0} \bigg\{f^0(t,i^0,\alpha^0, x) + \sum_{j^0\neq i^0} (u_{j^0} - u_{i^0}) q^0(t,i^0,j^0, \alpha^0,x) \\
&\;\;+ (1-\sum_{k=1}^{M}x_k)  \sum_{k=1}^{M-1}p_k  q(t,M,k,\phi(t, M, i^0, x), i^0, \alpha^0, x)+ \sum_{k,l=1}^{M-1}p_l x_k q(t,k,l,\phi(t,k,i^0,x),i^0,\alpha^0, x)\bigg\},\;\;\;\forall t\le T
\end{align*}
Then we have:
\begin{align*}
H^{0*}(i^0,t,x,u,d,p) &= H_0{*}(i^0,t,x,u,d,p)  = H^0(i^0,t,x,u,d,p), \;\text{if}\;\;\;t<T\\
H^{0*}(i^0,T,x,u,d,p) &= \max \{ (g^0(i^0,x) - u_{i^0}), \;H^{0*}(i^0,t,x,u,d,p) \}\\
H^{0}_*(i^0,T,x,u,d,p) &= \min \{ (g^0(i^0,x) - u_{i^0}), \;H^{0*}(i^0,t,x,u,d,p) \}
\end{align*}
\end{lemma}
From the proof of Theorem \ref{hjbtheomajor}, we see immediately that $V^0$ is a viscosity subsolution (resp. supersolution) of $H^{0,*} = 0$ (resp. $H^{0}_* = 0$). By Lemma \ref{lemprooflem2}, $V^{0*}$ (resp. $V^{0}_*$) is a viscosity subsolution of $H^{0,*} = 0$ (resp. supersolution of $H^{0}_* = 0$). Indeed, following exactly the proof of Theorem \ref{compmajor}, we can show that a viscosity supersolution of $H^{0}_* = 0$ is greater than a viscosity subsolution of $H^{0,*} = 0$. By definition of $\limsup$ and $\liminf$, we have $V^{0,*} \ge V^{0}_*$. It follows that $V^0 \le V^{0}_* \le V^{0,*} \le V^0$, and therefore we have $\limsup V_n^0 = \liminf H_n^0 = V^0$. Then we obtain the uniform convergence of $V_n^0$ to $V^0$ following Remark 6.4 in \cite{lions}.

\subsection{Proof of Proposition \ref{majorNproperty} \& \ref{majorvalueodeproperty}}
We use the similar techniques as in \cite{GomesMohrSouza_continuous} where the author provides gradient estimates for N-player game without major player. Let us first remark that the system of ODEs (\ref{odemajorpayoff}) can be written in the following form:
\[
-\dot\theta_{m}(t) = f_m(t) + \sum_{m'\neq m} a_{m'm}(t)(\theta_{m'} - \theta_{m}),\;\;\;\theta_m(T) = b_m
\]
where we denote the index $m:=(i^0,x)\in\{1,\dots,M^0\}\times\mathcal{P}^N$ and we notice that $a_{m'm}\ge 0$ for all $m'\neq m$. This can be further written in the compact form:
\begin{equation}\label{generalode}
-\dot\theta(t) = f(t) + M(t)  \theta,\;\;\;\theta(T) = b
\end{equation}
where $M$ is a matrix indexed by $m$, with all off-diagonal entries being positive and the sum of every row equals $0$. Define $\|\cdot\|$ to be the uniform norm of a vector: $\|b\| := \max_{m} |b_m|$. Instead of proving (i) in Proposition \ref{majorNproperty}, we prove a more general result, which is a consequence of Lemma 4 and Lemma 5 in \cite{GomesMohrSouza_continuous}. 
\begin{lemma}\label{odebound}
Let $\phi$ be a solution to the ODE (\ref{generalode}). Assume $f$ is bounded then we have
\[
\|\theta(t)\| \le \int_t^T \|f(s)\|ds  + \|b\|
\]
\end{lemma}
\begin{proof}
For any $t\le s \le T$ we denote the matrix $K(t,s)$ as the solution of the following system:
\[
\frac{dK(t,s)}{dt} = -M(t) K(t,s), \;\;\; K(s,s) = I
\]
where $I$ stand for the identity matrix. Now for $\theta$, we clearly have
\[
-\dot\theta(s) = f(s) + M(s)  \theta(s) \le \|f(s)\| e + M(s)  \theta (s)
\]
where we denote $e$ to be the vector with all the components equal to $1$. Then using Lemma 5 in \cite{GomesMohrSouza_continuous}, we have
\[
-K(t,s)\dot\theta(s) \le K(t,s)  M(t)  \theta (s) + \|f(s)\| K(t,s)  e
\]
Note that $K(t,s)\dot\theta(s) + K(t,s)  M(t)  \theta (s) = \frac{d}{ds}K(t,s)\theta(s)$. We integrate the above inequality between $t$ and $T$ to obtain:
\[
\theta(t) \le K(t,T) b + \int_{t}^{T} \|f(s)\| K(t,s)  e \;ds
\]
Now using Lemma 4 in \cite{GomesMohrSouza_continuous}, we have $\|K(t,T) b\| \le \|b\|$ and $\|K(t,T) e\| \le \|e\|=1$. This implies that:
\[
\max_{m} \theta_m(t) \le \|b\| + \int_t^T \|f(s)\|ds
\]
Indeed starting from the inequality $\dot\theta(s) \ge -\|f(s)\| e +  M(s)  \phi (s)$ and going through the same steps we obtain:
\[
\min_{m} \theta_m(t) \ge -\|b\| - \int_t^T \|f(s)\|ds
\]
The desired inequality follows.
\end{proof}
Now we turn to the proof of Proposition \ref{majorNproperty}. Let $\theta$ be the unique solution to the system of ODEs (\ref{odemajorpayoff}). Recall the notation $e_{ij}:=\mathbbm{1}_{j\neq M}e_j-\mathbbm{1}_{i\neq M}e_i$. For any $k\neq l$ and define $z(t,i^0,x,k,l):=\theta(t,i^0,x+\frac{1}{N} e_{kl}) - \theta(t,i^0,x)$. Then $z(t,\cdot,\cdot,\cdot,\cdot)$ can be viewed as a vector indexed by $i^0, x, k, l$. Substracting the ODEs satisfied by $\theta(t,i^0,x+\frac{1}{N}  e_{kl})$ and $\theta(t,i^0,x)$, we obtain that $z$ solves the following system of ODEs:
\begin{equation}\label{odegradient}
\begin{aligned}
-\dot z(t,i^0,x,k,l) = & F(t,i^0,x,k,l) + \sum_{j^0, j^0\neq i^0}(z(t,j^0,x,k,l) -  z(t,i^0,x,k,l)) q^0_{\phi^0}(t, i^0, j^0,x)\\
& + \sum_{(i,j),j\neq i} (z(t,i^0,x+\frac{1}{N} e_{ij}, k, l) - z(t,i^0,x,k,l))  N x_i q_{\phi}(t,i,i^0,x)\\ 
 z(T, i^0, x, k, l) =&g^0(i^0, x+\frac{1}{N} e_{kl}) - g^0(i^0, x)
\end{aligned}
\end{equation}
Where we have defined the $F$ to be:
\begin{align*}
 F(t,i^0,x,k,l):=&\; f^0_{\phi^0}(t,i^0,x+\frac{1}{N} e_{kl}) - f^0_{\phi^0}(t, i^0, x)\\
&+\sum_{j^0,j^0\neq i^0} [\theta(t,j^0,x+\frac{1}{N}e_{kl})- \theta(t,i^0,x+\frac{1}{N}e_{kl})][q^{0}_{\phi^0}(t,i^0,j^0,x+\frac{1}{N}e_{kl}) - q^{0}_{\phi^0}(t, i^0, j^0,x)]\\
& + \sum_{(i,j),j\neq i} [\theta (t,i^0, x+\frac{1}{N}e_{ij}+\frac{1}{N}e_{kl}) - \theta (t,i^0, x+\frac{1}{N}e_{kl})]\\
&\;\;\;\;\;\;\;\;\;\;\;\;\;\;\; \times N[(x+\frac{1}{N}e_{kl})_{i} q_{\phi}(t,i,j,i^0, x+\frac{1}{N}e_{kl}) - x_i q_{\phi}(t,i,j,i^0, x)]
\end{align*}
Then by the uniform bound provided in Lemma \ref{odebound}, together with the Lipschitz property of $f^0, g^0, q^0, q, \alpha,\beta$, we have:
\[
\|F(t)\| \le \frac{L}{N} + 2(T\|f^0\|_{\infty} + \|g^0\|_{\infty})  \frac{L}{N}  + \|z(t)\|\cdot L,\;\;\;\;\|z(T)\|\le \frac{L}{N}
\]
Where $L$ is a generic Lipschitz constant. Now we are ready to apply Lemma \ref{odebound} to (\ref{odegradient}). We have for all $t\le T$
\[
\|z(t)\|\le \frac{L}{N} + \int_{t}^{T} (\frac{C}{N} + L \|z(s)\|) ds
\]
Where $C$ is a constant only depending on $T, L, \|f^0\|_{\infty}, \|g^0\|_{\infty}$. Finally the Gronwall's inequality allows to conclude.

The proof of Proposition \ref{majorvalueodeproperty} is similar. We consider the solution $\theta$ to the ODE (\ref{odemajorvalue}) and we  keep the notation $z(t,i^0,x,k,l)$ as before. For $v\in\mathbb{R}^{M^0}$, $x\in\mathcal{P}$, $t\le T$ and $i^0 = 1,\dots,M^0$, denote:
\[
h^0(t,i^0,x,v) := \inf_{\alpha^0 \in A^0} \{f^0(t,\alpha^0,i^0,x) + \sum_{j^0\neq i^0} (v_{j^0} - v_{i^0}) q^0(t,i^0,j^0,\alpha^0,x)\}
\]
Then for all $x,y \in \mathcal{P}$, $u,v \in \mathbb{R}^{M^0}$, using the Lipschitz property of $f^0$ and $q^0$ and the boundedness of $q^0$, we have
\begin{equation}\label{esth}
|h^0(t,i^0,x,v) - h^0(t,i^0,y,u)| \le L|x-y| + 2(M^0-1) \max\{\|u\|, \|v\|\}  \|x - y\| + C^q  2(M^0-1)  \|v-u\|
\end{equation}
Substracting the ODEs satisfied by $\theta(t,i^0,x+\frac{1}{N}  e_{kl})$ and $\theta(t,i^0,x)$, we obtain that $z$ solves the following system of ODEs:
\begin{align*}
-\dot z(t,i^0,x,k,l) = & F(t,i^0,x,k,l) + \sum_{(i,j),j\neq i} (z(t,i^0,x+\frac{1}{N} e_{ij}, k, l) - z(t,i^0,x,k,l))  N x_i q_{\phi}(t,i,j,i^0,x)\\ 
 z(T, i^0, x, k, l) =&g^0(i^0, x+\frac{1}{N} e_{kl}) - g^0(i^0, x)
\end{align*}
where $F$ is given by:
\begin{align*}
 F(t,i^0,x,k,l):=&h(t,i^0,x+\frac{1}{N}e_{kl},\theta (t,\cdot, x+\frac{1}{N}e_{kl})) - h(t,i^0,x,\theta (t,\cdot, x))\\
& + \sum_{(i,j),j\neq i} [\theta (t,i^0, x+\frac{1}{N}e_{ij}+\frac{1}{N}e_{kl}) - \theta (t,i^0, x+\frac{1}{N}e_{kl})]\\
&\;\;\;\;\;\;\;\;\;\;\;\;\;\;\; \times N[(x+\frac{1}{N}e_{kl})_{i} q_{\phi}(t,i,j,i^0, x+\frac{1}{N}e_{kl}) - x_i q_{\phi}(t,i,j,i^0, x)]
\end{align*}
By the estimation (\ref{esth}), the uniform bound provided in Lemma \ref{odebound}, together with the Lipschitz property of $f^0, g^0, q^0, q, \phi^0,\phi$, we have:
\[
\|F(t)\| \le \frac{1}{N}(L + 2(M^0 - 1)(T\|f\|^0_{\infty} + \|g\|^0_{\infty})) + 2C_q(M^0-1)\|z(t)\| + M( M - 1)( L_\phi L + C_q )\|z(t)\|
\]
We apply Lemma \ref{odebound} to obtain:
\begin{align*}
\|z(t)\|\le&\;\; \frac{L}{N} + \int_{t}^{T} (M( M - 1)( L_\phi L + C_q )+  2C_q(M^0-1))\|z(s)\| + \frac{1}{N}( L + 2(M^0 - 1)(T\|f\|^0_{\infty} + \|g\|^0_{\infty})) ds\\
:=&\;\;\frac{C_0 + C_1 T + C_2 T^2}{N} + \int_t^T (C_3 + C_4 L_\phi) \|z(s)\| ds
\end{align*}
The Gronwall's inequality allows to conclude.

\bibliography{games}

\begin{thebibliography}{10}

\bibitem{BarlesSouganidis}
{\sc G.~Barles and P.~Souganidis}, {\em Convergence of approximation schemes
  for fully nonlinear second order equations}, Asymptotic Analysis, 4 (1991),
  pp.~271--283.

\bibitem{BensoussanChauYam}
{\sc A.~Bensoussan, M.~Chau, and S.~Yam}, {\em Mean field games with a
  dominating player}, tech. rep., 2013.

\bibitem{BrianiCamilliZidani}
{\sc A.~Briani, F.~Camilli, and H.~Zidani}, {\em Approximation schemes for
  monotone systems of nonlinear second order partial differential equations:
  convergence result and error estimate}, Differential Equations and
  Applications, 4 (2012), pp.~297--317.

\bibitem{CarmonaDelarue_book_I}
{\sc R.~Carmona and F.~Delarue}, {\em Probabilistic Theory of Mean Field Games:
  vol. I, Mean Field FBSDEs, Control, and Games}, Stochastic Analysis and
  Applications, Springer Verlag, 2017.

\bibitem{CarmonaDelarue_book_II}
\leavevmode\vrule height 2pt depth -1.6pt width 23pt, {\em Probabilistic Theory
  of Mean Field Games: vol. II, Mean Field Games with Common Noise and Master
  Equations}, Stochastic Analysis and Applications, Springer Verlag, 2017.

\bibitem{CarmonaWang_LQ}
{\sc R.~Carmona and P.~Wang}, {\em An alternative approach to mean field games
  with major and minor players with application to a herder model}, tech. rep.,
  Princeton University, 2016.

\bibitem{CarmonaZhu}
{\sc R.~Carmona and G.~Zhu}, {\em A probabilistic approach to mean field games
  with major and minor players}, Annals of Applied Probability, 26 (2014),
  pp.~1535--1580.

\bibitem{EthierKurtz}
{\sc S.~Ethier and T.~Kurtz}, {\em Markov Processes: Characterization and
  Convergence}, Wiley, 2005.

\bibitem{FlemingSoner}
{\sc W.~Fleming and M.~Soner}, {\em Controlled Markov Processes and Viscosity
  Solutions}, Springer Verlag, 2010.

\bibitem{GomesMohrSouza_continuous}
{\sc D.~Gomes, J.~Mohr, and R.~Souza}, {\em Continuous time finite state mean
  field games}, Applied Mathematics and Optimization, 68 (2013), pp.~99 --143.

\bibitem{Gueant_tree}
{\sc O.~Gu{\'{e}}ant}, {\em From infinity to one: The reduction of some mean
  field games to a global control problem}, Cahier de la Chaire Finance et
  D{\'{e}}veloppement Durable, 42 (2011).

\bibitem{Gueant_congestion}
\leavevmode\vrule height 2pt depth -1.6pt width 23pt, {\em Existence and
  uniqueness result for mean field games with congestion effect on graphs},
  Applied Mathematics \& Optimization, 72 (2015), pp.~291 -- 303.

\bibitem{Huang}
{\sc M.~Huang}, {\em Large-population lqg games involving a major player: the
  nash equivalence principle}, {SIAM} Journal on Control and Optimization, 48
  (2010), pp.~3318--3353.

\bibitem{JaimungalNourian}
{\sc S.~Jaimungal and M.~Nourian}, {\em Mean-field game strategies for a
  major-minor agent optimal execution problem}, tech. rep., University of
  Toronto, March 15, 2015.

\bibitem{KolokoltsovBensoussan}
{\sc V.~Kolokolstov and A.~Bensoussan}, {\em Mean-field-game model for botnet
  defense in cyber-security}, tech. rep., {\tt
  http://arxiv.org/abs/1511.06642}, 2015.

\bibitem{lions}
{\sc P.~Lions}, {\em Th\'eorie des jeux \`a champs moyen et applications}.
\newblock Lectures at the {C}oll\`ege de {F}rance.
  http://www.college-de-france.fr/default/EN/all/equ\_der/cours\_et\_seminaires.htm,
  2007-2008.

\bibitem{NourianCainesMalhame}
{\sc M.Nourian, P.~Caines, and R.~Malham{\'e}}, {\em Mean field analysis of
  controlled {C}ucker-{S}male type flocking: Linear analysis and perturbation
  equations}, in Proceedings of the 18th {IFAC} World Congress, {M}ilan,
  {A}ugust 2011, 2011, pp.~4471--4476.

\bibitem{NguyenHuang1}
{\sc S.~Nguyen and M.~Huang}, {\em Linear-quadratic-{G}aussian mixed games with
  continuum-parametrized minor players}, SIAM Journal on Control and
  Optimization,  (2012).

\bibitem{NguyenHuang2}
\leavevmode\vrule height 2pt depth -1.6pt width 23pt, {\em Mean field {LQG}
  games with mass behavior responsive to a major player}, in 51th {IEEE}
  Conference on Decision and Control, 2012.

\bibitem{NourianCaines}
{\sc M.~Nourian and P.~Caines}, {\em $\epsilon$-nash mean field game theory for
  nonlinear stochastic dynamical systems with major and minor agents}, tech.
  rep., 2013.

\bibitem{swart}
{\sc J.~Swart and A.~Winter}, {\em Markov processes: theory and examples},
  tech. rep., Mimeo, https://www. unidue. de/.../Markovprocesses/sw20, 2013.

\end{thebibliography}

\end{document}